\documentclass[10pt]{amsart}
\usepackage{graphicx,enumitem}
\usepackage{subfigure}

\usepackage{color,mathrsfs,epsfig}
\usepackage{amsmath}
\usepackage{amsfonts}
\usepackage{amssymb}
\usepackage{amsthm}
\usepackage{latexsym}

\newtheorem{definition}{Definition}[section]
\newtheorem{theorem}{Theorem}[section]

\newtheorem{remark}{Remark}[section]
\newtheorem*{maintheorem*}{Main Theorem}
\allowdisplaybreaks
\numberwithin{equation}{section}

\renewcommand{\i}{\ifmmode\mathit{\mathchar"7010 }\else\char"10 \fi}
\renewcommand{\j}{\ifmmode\mathit{\mathchar"7011 }\else\char"11 \fi}
\newcommand{\R}{\mathbb{R}}

\newcommand{\D}{\mathcal{D}}

\newcommand{\Dx}{\Delta x}

\newcommand{\Dt}{\Delta t}

\newcommand{\U}{{\bf U}}

\newcommand{\W}{{\bf W}}
\newcommand{\F}{{\bf F}}

\newcommand{\V}{{\bf V}}

\newcommand{\A}{\mathcal A}
\newcommand{\B}{\mathcal B}
\newbox\bokstav
\newdimen\hoyde
\def\bgl{{\hbox{$\left\lbrack\vbox to 8.5pt{}\right.\nOspace$}}}
\def\Bgl{{\hbox{$\left\lbrack\vbox to 11.5pt{}\right.\nOspace$}}}
\def\bggl{{\hbox{$\left\lbrack\vbox to 14.5pt{}\right.\nOspace$}}}
\def\Bggl{{\hbox{$\left\lbrack\vbox to 17.5pt{}\right.\nOspace$}}}
\def\bgr{{\hbox{$\left\rbrack\vbox to 8.5pt{}\right.\nOspace$}}}
\def\Bgr{{\hbox{$\left\rbrack\vbox to 11.5pt{}\right.\nOspace$}}}
\def\bggr{{\hbox{$\left\rbrack\vbox to 14.5pt{}\right.\nOspace$}}}
\def\Bggr{{\hbox{$\left\rbrack\vbox to 17.5pt{}\right.\nOspace$}}}
\def\nOspace{\nulldelimiterspace=0pt \mOth}
\def\mOth{\mathsurround=0pt}
\def\Ljmp{\mathopen{\lbrack\!\lbrack}}
\def\Rjmp{\mathclose{\rbrack\!\rbrack}}
\def\bgLjmp{\mathopen{\bgl\mskip-6mu\bgl}}
\def\bgRjmp{\mathclose{\bgr\mskip-6mu\bgr}}
\def\BgLjmp{\mathopen{\Bgl\!\!\Bgl}}
\def\BgRjmp{\mathclose{\Bgr\!\!\Bgr}}
\def\bggLjmp{\mathopen{\bggl\!\!\bggl}}
\def\bggRjmp{\mathclose{\bggr\!\!\bggr}}
\def\BggLjmp{\mathopen{\Bggl\!\!\Bggl}}
\def\BggRjmp{\mathclose{\Bggr\!\!\Bggr}}
\def\jmp#1{
\setbox\bokstav=\hbox{$ \left. #1\right. $}
\hoyde=\ht\bokstav 
\advance\hoyde by \dp\bokstav
\hbox{$
        \ifinner
                \ifdim\hoyde<10pt
                   \Ljmp #1 \Rjmp%
                \else
                   \ifdim\hoyde <11pt
                      \Ljmp #1 \Rjmp%
                   \else
                      \ifdim\hoyde <14pt
                          \bgLjmp #1 \bgRjmp%
                      \else
                          \ifdim\hoyde <20pt
                             \BgLjmp #1 \BgRjmp%
                          \else
                              \bggLjmp #1 \bggRjmp%
                          \fi
                      \fi
                   \fi
                \fi
        \else
                \ifdim\hoyde<8.5pt
                   \Ljmp #1 \Rjmp%
                \else
                   \ifdim\hoyde <11.5pt
                      \bgLjmp #1 \bgRjmp%
                   \else
                      \ifdim\hoyde <14.5pt
                          \BgLjmp #1 \BgRjmp%
                      \else
                          \ifdim\hoyde <17.5pt
                             \bggLjmp #1 \bggRjmp%
                          \else
                              \BggLjmp #1 \BggRjmp%
                          \fi
                      \fi
                   \fi
            \fi
        \fi
$}
}

\newcounter{asnr}

{\ifnum\value{asnr}=0 \stepcounter{asnr} 
  \begin{enumerate}[label=\textbf{A}.\arabic{enumi}]
    \else
    \begin{enumerate}[label=\textbf{A}.\arabic{enumi},resume] \fi}
{\end{enumerate}}

\begin{document}

\date{\today}

\title[Accurate numerical schemes ]{Accurate numerical schemes for approximating \\initial-boundary value problems \\
for systems of conservation laws.}

\author[Siddhartha Mishra]{Siddhartha Mishra} \address[Siddhartha
Mishra]{\newline Center of Mathematics for Applications (CMA) \newline
 University of Oslo, P.O. Box-1053, \newline 
 Blindern, Oslo - 0316, Norway} \email[]{smishra@sam.math.ethz.ch}

\author[Laura V. Spinolo]{Laura V. Spinolo}\address[Laura V. Spinolo]
{\newline IMATI-CNR, \newline I-27100, Pavia, Italy and  
\newline Universit\"at Z\"urich,  \newline CH-8057 Z\"urich, Switzerland}
\email[]{spinolo@imati.cnr.it}


\maketitle

\begin{abstract}
Solutions of initial-boundary value problems for systems of conservation laws depend 
on the underlying viscous mechanism, namely different viscosity operators lead to different limit solutions. 
Standard numerical schemes for approximating conservation laws do not take into account this fact and converge to 
solutions that are not necessarily physically relevant. We design numerical schemes that incorporate 
explicit information about the underlying viscosity mechanism and approximate the physically relevant solution. 
Numerical experiments illustrating the robust performance of these schemes are presented.
\end{abstract}

\section{Introduction}
\label{sec:intro}
Many problems in physics and engineering are modeled by systems of conservation laws
\begin{equation}
\label{eq:cl}
\begin{aligned}
\U_t + \F(\U)_x &= 0.
\end{aligned}
\end{equation}
Here, $\U:\Omega \times \R_+ \to \R^m$ is the vector of unknowns and $\F:\R^m \to \R^m$ is the flux vector. The spatial domain is a set $\Omega \subset \R$. The above equations are augmented with initial data. If $\Omega$ is a bounded domain, the conservation laws are augmented with suitable boundary conditions. Examples of conservation laws include the shallow water equations of oceanography, the Euler equations of gas dynamics and the equations of MagnetoHydroDynamics (MHD).

We assume that the system of conservation laws is strictly hyperbolic, i.e. the eigenvalues of the Jacobian
 {\color{black}matrix} $\F_{\U}$ are real 
{\color{black} and distinct. Also, we assume that all the eigenvalues of
$\F_{\U}$ are bounded away from $0$:
\begin{equation}
\label{eq:ncharb}
        \lambda_1 (\U) < \dots < \lambda_k (\U) < -d < 0 < d < \lambda_{k+1} (\U) < \dots < \lambda_m (\U)
\end{equation}
for some positive constant $d>0$ and some integer $k < m$.

 It is well known (see Dafermos~\cite[Chapter 6]{DAF1}) that in general solutions of \eqref{eq:cl} form discontinuities (shock waves, contact discontinuities)
in finite time even when the initial data are smooth. Hence, solutions of \eqref{eq:cl} are defined in 
the sense of distributions.

In general, the distributional solution of a given Cauchy or initial-boundary value problem is not unique and hence various 
admissibility conditions have been introduced in the attempts at selecting a unique solution, see the book by 
Dafermos~\cite[Chapters 4 and 8]{DAF1} 
for an extended discussion. These approaches often involve the celebrated \emph{entropy condition}, which can be formulated as follows: assume that 
system~\eqref{eq:cl} admits an entropy-entropy flux pair, namely   
there exists a convex function $S:\R^m \to \R$ and a function $Q:\R^m \to \R$ 
such that 
\begin{equation}
\label{eq:ent1}
Q_{\U} = S_\U \F_\U,
\end{equation} 
where  $S_{\U}$ and $Q_{\U}$ denote the gradients of the function $S$ and $Q$, respectively. 
A distributional solution $U$ satisfies the entropy admissibility condition if the following inequality holds in the 
sense of distributions: 
\begin{equation}
\label{eq:ent}
S(\U)_t + Q(\U)_x \leq 0.
\end{equation}
Here, two remarks are in order:  first, in general physical systems admit entropy-entropy flux pairs. Second, systems of conservation laws like \eqref{eq:cl} are derived by neglecting small scale effects like diffusion. Inclusion of these small effects in \eqref{eq:cl} results in the mixed hyperbolic-parabolic system:
\begin{equation}
\label{eq:vcl}
\U^{\epsilon}_t + \F(\U^{\epsilon})_x = \epsilon\left(\B(\U^{\epsilon})\U^{\epsilon}_x\right)_x.
\end{equation}
Here, $\epsilon$ is a (small) viscosity parameter and  $\B:\R^m \to \R^{m\times m}$ is the viscosity matrix. For example, the Navier-Stokes equations are a viscous regularization of the Euler equations of gas dynamics. {\color{black}In physical systems, the entropy admissibility criterion is consistent with the zero small scale effects limit}, namely one can show that, if the solutions of the viscous approximation~\eqref{eq:vcl} 
converge in a strong enough topology, then the limit satisfies~\eqref{eq:ent}.

We now focus on the initial-boundary value problem obtained by coupling the system of conservation laws~\eqref{eq:cl} with 
the Cauchy and Dirichlet data 
\begin{equation}
\label{eq:cl:bc}
        \U(x, 0) = \U_0(x), \quad x \in \Omega = (X_l, \infty) \qquad \U( X_l, t) = \bar \U (t), \quad t \in \R_+.
\end{equation}
The study of the initial-boundary value problem poses additional difficulties as compared to the study of the Cauchy problem: 
first, the problem~\eqref{eq:cl}-\eqref{eq:cl:bc} is, 
in general, ill posed (i.e. it possesses no solutions) unless additional conditions are imposed on the data $\bar \U$. Possible 
admissibility criteria on $\bar \U$ are discussed in Dubois and LeFloch~\cite{DLF1}. 
 
Another additional difficulty one has to tackle when studying initial-boundary value problems is the following: consider the viscous approximation~\eqref{eq:vcl} {\color{black}coupled with the initial and
boundary data 
\begin{equation}
\label{e:vcl:bd}
       \U^{\epsilon}(x, 0) = \U_0(x), \quad x \in \Omega = (X_l, \infty) \qquad \U^{\epsilon}( X_l, t) = \U_l (t), \quad t \in \R_+.
\end{equation}
Assume that the initial-boundary value problem~\eqref{eq:vcl},~\eqref{e:vcl:bd} is well posed (this is not always the case in the case when $\B$ is singular) and that for $\epsilon \to 0^+$ the solutions converge 
in a suitable topology to a limit $\U$. In general, because of boundary layer phenomena, $\U$ \emph{may not} satisfy 
the boundary condition $\U_l (t)$ pointwise. Dubois and LeFloch~\cite{DLF1} showed that, if 
the solutions of the viscous approximation~\eqref{eq:vcl} converge as $\epsilon \to 0^+$ 
to a solution of the initial-boundary 
problem~\eqref{eq:cl},~\eqref{eq:cl:bc} in a sufficiently strong topology, 
then the following inequality holds:
\begin{equation}
    \label{eq:DLF}
Q(\overline{\U}(t)) - Q(\U_l(t)) - \langle S_{\U}(\U_l(t)),\left(\F(\overline{\U}(t)) -\F(\U_l(t))\right) \rangle \leq 0.
\end{equation}
Here $\langle \cdot, \cdot \rangle$ denotes the standard scalar product in $\R^m$.

A further difficulty in the study of initial-boundary value problems was pointed out in the works by Gisclon and Serre~\cite{G1, GS1}: they} showed that the limit of the viscous approximation~\eqref{eq:vcl} 
\emph{depends} on the underlying viscosity mechanism. In other words, the limit 
of~\eqref{eq:vcl} in general changes if one changes the viscosity matrix $\B$.

As an example, we consider the linearized shallow water equations \eqref{eq:lsws} with initial data \eqref{eq:lswinit} and boundary data \eqref{eq:ldir}. The system is a \emph{linear}, strictly hyperbolic, $2 \times 2$ system and is the simplest possible problem that can be considered in this context. We consider two different viscosity operators: an \emph{artificial} uniform (Laplacian) viscosity \eqref{eq:lswslap} and the \emph{physical} 
eddy viscosity \eqref{eq:lswsed}. The resulting limit solutions are shown in the left of figure \ref{fig:1}. As shown in the figure, there is a significant difference in solutions (near the boundary) corresponding to different viscosity 
operators.  

An extended discussion concerning the initial boundary value problem for 
systems of conservation laws and its viscous approximation can be found in
the books by Serre~\cite[Chapters 14 and 15]{Serre:b1, Serre:b2}, while we refer to the lecture notes by Serre~\cite{Serre:lecturenotes} 
and to the rich bibliography therein for the theoretical treatment of the discrete approximation of viscous shock profiles. 
To conclude, we stress that analytically establishing the convergence $\epsilon \to 0^+$ for~\eqref{eq:vcl} is still an open problem in the general case, but results 
are available in more specific cases: in particular, Gisclon~\cite{G1} showed local-in-time convergence in the case when $\B$ is invertible and, by extending the analysis in 
Bianchini and Bressan~\cite{BIABRE:ANNALS},
Ancona and Bianchini~\cite{ANBIA} proved global-in-time convergence in the case when $\B$ is the identity.
\subsection{Numerical schemes}
Numerical schemes play a very important role in the study of system of conservation laws. Conservative finite difference (finite volume) 
methods are among the most popular discretization frameworks for \eqref{eq:cl}. See the book by LeVeque~\cite{LEV1} for an extended discussion. {\color{black}Given the real numbers 
$X_l <X_r$}, we discretize the computational 
domain $[X_l,X_r]$ by $N+1$ equally spaced points $x_{j+1/2} = X_l + j \Dx$ with $X_{1/2} = X_l$ and with mesh size $\Dx$ and we set $x_{j} = \frac{x_{j-1/2} + x_{j+1/2}}{2}$. Time is discretized with a time step $\Dt^n$. The mesh size and time step are 
related by a standard CFL condition.

The aim is to approximate cell averages $\U^n_j$ of the unknown $\U$ in the cell ${\mathcal C}_j = [x_{j-1/2},x_{j+1/2})$ at 
time $t^n$ by the scheme
\begin{equation}
\label{eq:fvs}
\U^{n+1}_j = \U^n_{j} - \frac{\Dt^n}{\Dx}\left(\F^n_{j+1/2} - \F^n_{j-1/2}\right).
\end{equation}
Here, $\F^n_{j+1/2} = \F(\U^n_j,\U^n_{j+1})$ is the numerical flux. The numerical flux is obtained by solving 
(approximately) the Riemann problem for \eqref{eq:cl} with the states $\U^n_j$ and $\U^{n+1}_j$. 

Following \cite{DLF1,LEV1}, the Dirichlet boundary conditions at $X=X_l$ are imposed by setting in the \emph{ghost} cell $[x_{-1/2},x_{1/2}]$:
\begin{equation}
\label{eq:bcnum}
\U^n_0 = \U_l(t^n).
\end{equation}

\begin{figure}[htbp]
\centering
\subfigure[Viscous profile]{\includegraphics[width=0.48\linewidth]{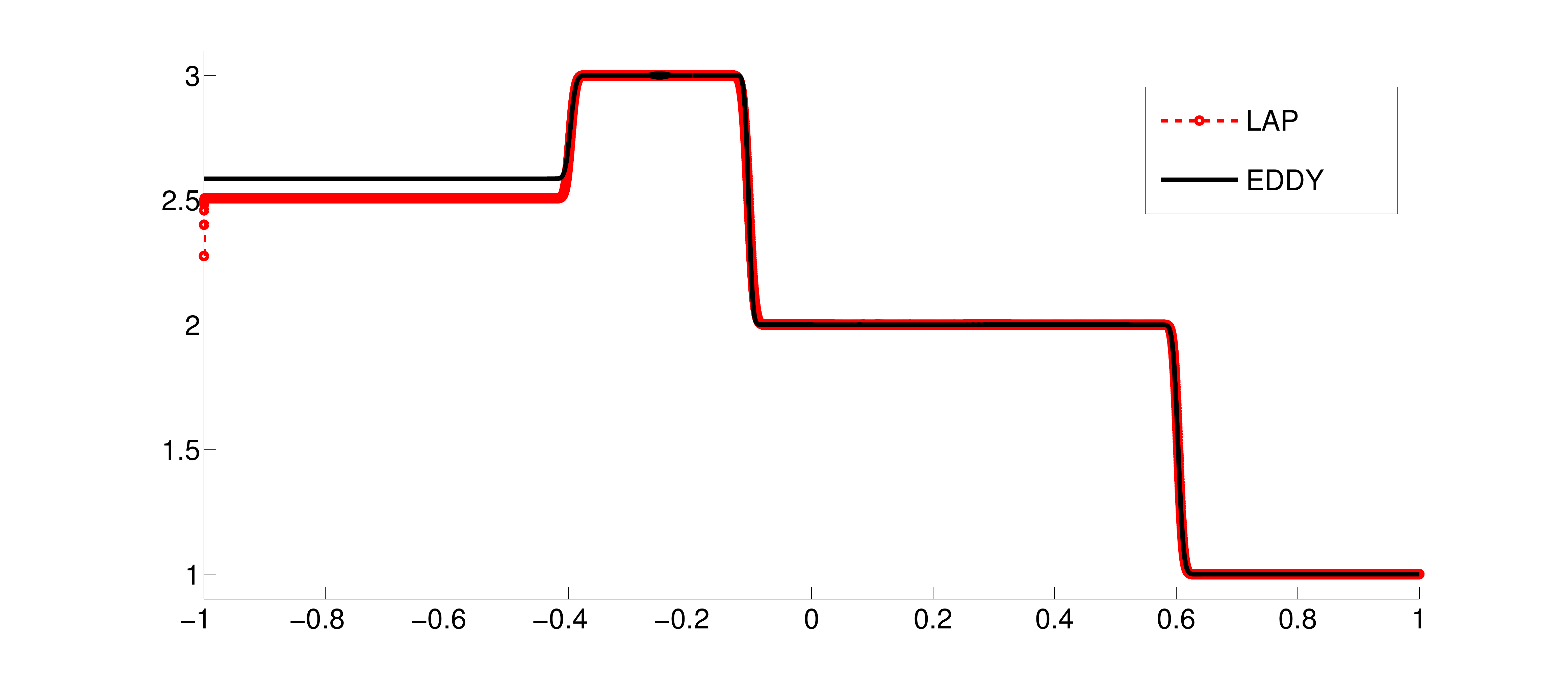}}
\subfigure[Roe scheme]{\includegraphics[width=0.48\linewidth]{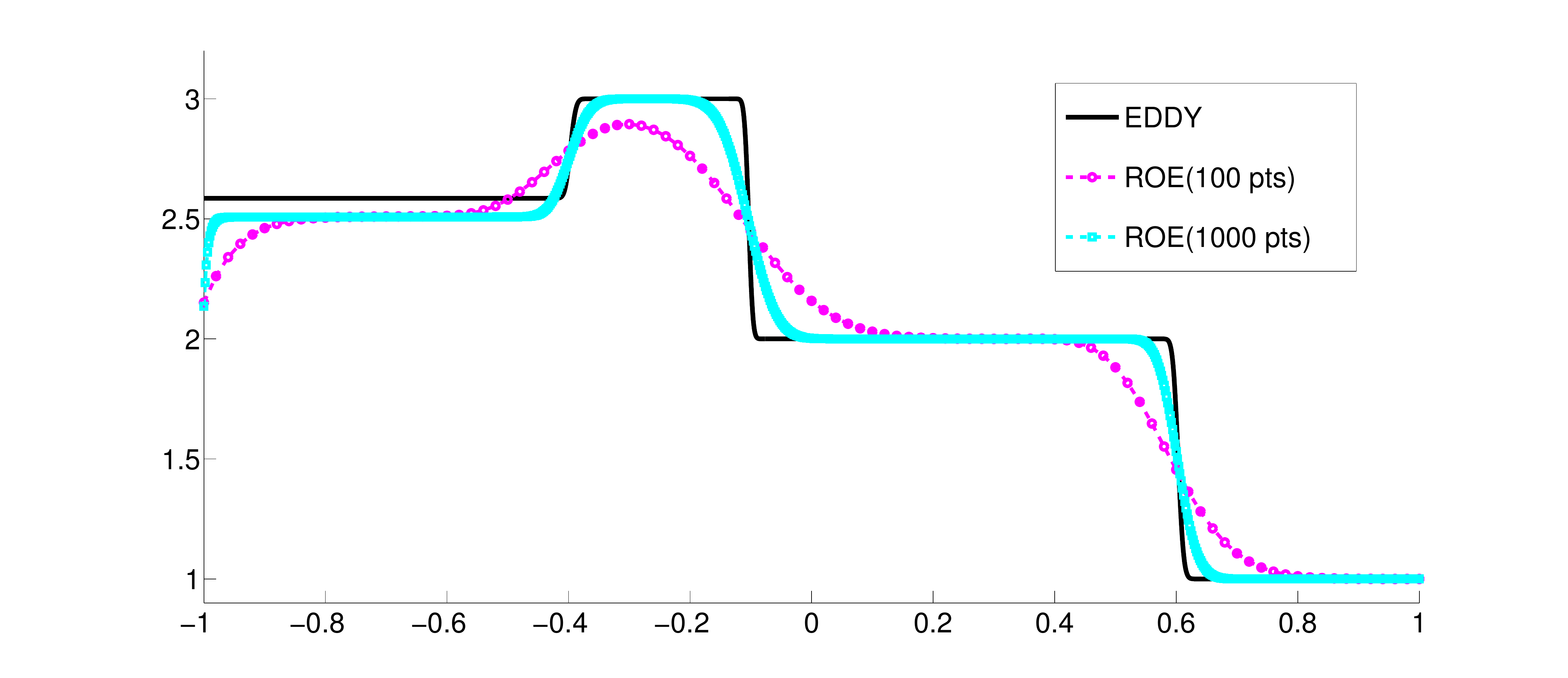}}
\caption{Left: The limit viscous profile for the linearized shallow water equations \eqref{eq:lsws} with uniform (Laplacian) viscosity \eqref{eq:lswslap} and 
eddy viscosity \eqref{eq:lswsed}. Right: Approximate solutions with the Roe (Godunov) scheme for the linearized shallow water equations for the same data as the left figure.}
\label{fig:1}
\end{figure}
However, standard numerical schemes may not converge to the physically relevant solution of the initial boundary value problem for a system of conservation laws. We illustrate this by again considering the linearized shallow water equations \eqref{eq:lsws} with initial data 
\eqref{eq:lswinit} and boundary data \eqref{eq:ldir}. The results with a standard Roe (Godunov) scheme for this linear system are presented in figure \ref{fig:1}, right. The figure clearly shows that the Roe scheme converges to a solution that is different from the physically relevant solution of the system, realized as a limit of the eddy viscosity approximation \eqref{eq:lswsed}. In fact, the solution converges to the limit of the \emph{artificial} uniform viscosity approximation \eqref{eq:lswslap}.

The problem with standard numerical schemes approximating the initial-boundary value problem \eqref{eq:cl} lies 
in the fact that they do not incorporate explicit information about the underlying viscous approximation~\eqref{eq:vcl}. The implicit \emph{numerical viscosity} added by such schemes may lead to the schemes converging to an incorrect solution.  This situation presents analogies with the numerical 
approximation of non-classical shocks (see LeFloch\cite{LF1}), 
non-conservative hyperbolic systems (see Castro, LeFloch, Munoz Ruiz and Pares~\cite{CLMP1}) and conservation laws with discontinuous coefficients (see Admiurthi, Mishra and Veerappa Gowda~\cite{AGSID1}}).  

Here, we design numerical schemes that incorporate explicit information about the underlying viscous operators. 
Consequently, these schemes approximate the physically relevant 
solutions of system of conservation laws. The schemes are based on the following two ingredients:
\begin{itemize}
\item [(i.)] An entropy conservative discretization of the flux $\F$ in \eqref{eq:cl} (see Fjordholm, Mishra and Tadmor~\cite{FMT4} and Tadmor~\cite{TAD1}).
\item [(ii.)] Numerical diffusion operators for \eqref{eq:cl} that are based on the underlying viscosity matrix $\B$ in \eqref{eq:vcl}.
\end{itemize}
We present both first- and second-order schemes that are shown (numerically) to converge to the physically relevant solution of the system of conservation laws. 

The rest of the paper is organized as follows: in Section~\ref{sec:tf}, we discuss the theoretical results concerning
the initial-boundary value problems~\eqref{eq:vcl} that will be used in the following sections. 
In particular, explicit solutions of the boundary value problem for a linear system are presented. In Section \ref{sec:schemes}, we present numerical schemes for the system of conservation laws \eqref{eq:cl} 
that converge to the physically relevant solution. Second-order schemes are discussed in Section \ref{sec:so}.

\section{Theoretical framework}
\label{sec:tf}
In the following section, we focus on the so-called \emph{boundary Riemann problem}, which is posed when 
the Cauchy and Dirichlet data for the mixed hyperbolic-parabolic system \eqref{eq:vcl} are two constant states, $\U_0(x) \equiv \U_0 \in \R^m$ and $\U_l (t) \equiv \U_l \in \R^m$. The solution of a general initial-boundary value problem can be build using the solutions of the boundary Riemann problem and standard Riemann problems in the interior, see Goodman~\cite{Goodman} and Sabl\'e-Tougeron~\cite{SableTougeron} (Glimm scheme) and Amadori~\cite{Amadori} (wave front-tracking algorithm).

\subsection{Linear case}
\label{s:lin}
\subsubsection{The solution of the Riemann problem in the linear case}
\label{s:linrie}
We start by recalling the solution of the Riemann problem obtained by coupling the linear 
system 
\begin{equation}
\label{eq:lin}
   \U_t + \A \U_x =0 \qquad \U \in \R^m
\end{equation}
with the initial datum
\begin{equation}
\label{eq:linrie}
      \U(0, x) =
\left\{
\begin{array}{ll}
       \U^- & x < 0  \\
       \U^+ & x  > 0. \\
\end{array}
\right. 
\end{equation}
In~\eqref{eq:lin}, $\A$ is a constant, strictly hyperbolic $m \times m$ matrix, and in~\eqref{eq:linrie} $\U^+$ and 
$\U^-$ are two given values in $\R^m$. 

Denote by $\lambda_1, \dots, \lambda_m$ the eigenvalues of $\A$ and by $R_1, \dots, R_m$ the corresponding 
right eigenvectors and consider the linear system 
\begin{equation}
\label{eq:linsyst}
  U^- + \sum_{i=1}^{m} \alpha_i R_i = U^+, 
\end{equation}
which by strict hyperbolicity admits a unique solution $(\alpha_1, \dots, \alpha_m)$. Then the solution 
of~\eqref{eq:lin}-\eqref{eq:linrie} is 
\begin{equation}
\label{eq:riesolver}  
\U (t, x) =
  \left\{
  \begin{array}{lll}
         U^-  & \text{if $x < \lambda_1 t$} & \\
         \displaystyle U^- + \sum_{i=1}^j \alpha_i R_i   &
         \text{if $\lambda_j t < x < \lambda_{j+1} t$}, & j =1, \dots, m-1 \\
         U^+   & \text{if $x > \lambda_m t$} & \\
  \end{array} 
  \right.
\end{equation} 
\subsubsection{The solution of the boundary Riemann problem in the linear case}
\label{s:blinrie}
We now consider the boundary Riemann problem obtained by coupling the linear mixed hyperbolic-parabolic system
\begin{equation}
\label{eq:linv}
   \U^{\epsilon}_t + \A \U^{\epsilon}_x =\epsilon \B \U^{\epsilon}_{xx} \qquad \U \in \R^m
\end{equation}
with the Dirichlet and Cauchy data,
\begin{equation}
\label{eq:linbrie}
      \U(t, 0) = \U_l, \qquad \U(0, x) = \U_0 (x), \qquad \forall \, t> 0, \; x>0, 
\end{equation}
and by taking the limit $\epsilon \to 0^+$.
The matrix $\A$ in~\eqref{eq:linv} is a constant $m \times m$ matrix satisfying~\eqref{eq:ncharb}, and $\B$ is another constant, 
$m \times m$ matrix which depends on the underlying physical model (we discuss explicit examples later in this paper). The data $U_l$ and $U_0$ in~\eqref{eq:linbrie} are constant states in $\R^m$. Note that, in general,
the problem~\eqref{eq:linv},\eqref{eq:linbrie} may be ill-posed if the matrix 
$\B$ is not invertible. However, to simplify the exposition in the present paper we always choose the data~\eqref{eq:linbrie} in such a way that it is well-posed.

As mentioned in the introduction, one of the main challenges coming from the presence of the boundary is the following: denote by $\U$ the limit 
$\epsilon \to 0^+$ of $U^{\epsilon}$, then in general the trace of $\U$ on the $t$-axis is not $\U_l$, 
\begin{equation}
\label{eq:trace}
     \bar \U \dot{=} \lim_{x \to 0^+} \U(t, x ) \neq \U_l.   
\end{equation}
More precisely, the relation between $\U_b$ and $\bar \U$ is the following: there is a function 
$\W: [0, + \infty[ \to \R^m$ satisfying 
   \begin{equation}
\label{eq:bl}
       \left\{
\begin{array}{lll}
             \B \dot{\W} = \A ( \U - \bar \U) \\
             W (0) = \U_l, \qquad
             \lim\limits_{y \to + \infty} \W(y) = \bar \U,
\end{array}
       \right.
\end{equation} 
where we denote by $\dot{\W}$ the first derivative of $\W(y)$. A function $\W$ satisfying~\eqref{eq:bl} is called a
 \emph{boundary layer}. 

Under the assumption~\eqref{eq:ncharb} and in the case when $\B$ is the identity, system~\eqref{eq:bl} admits a solution $W$ if and only if 
$$
  (\U_l - \bar \U ) \in \text{span} \langle R_1, \dots, R_k \rangle. 
$$    
We recall that by~\eqref{eq:ncharb}, $k$ is the number of negative eigenvalues 
of $\A$ and as in Section~\ref{s:linrie}, we denote by $R_1, \dots, R_k$ the corresponding eigenvectors. In general, when the matrix 
$\B$ is invertible, the system~\eqref{eq:bl} admits a solution if and only if $\U_l - \bar \U$ belongs to the stable space of 
						$\B^{-1} \A$ (i.e., to the subspace of $\mathbb R^m$ generated by the generalized eigenvectors associated to the eigenvalues of $\B^{-1} \A$ with strictly negative real part). Note that this space \emph{depends} on the matrix $\B$: this is the reason why, even in the simplest possible case
(linear system with an invertible viscosity matrix), the limit $\epsilon \to 0^+$ of~\eqref{eq:linv},~\eqref{eq:linbrie} 
\emph{depends}  on the choice of $\B$. 

In the case when $\B$ is not invertible, the analysis in Bianchini and 
Spinolo~\cite[Sections 4.2,4.3]{BIASPI:ARMA} guarantees that, in physical cases, if the initial-boundary 
value problem~\eqref{eq:linv},\eqref{eq:linbrie} is well-posed, then there are $k$ linearly independent vectors 
$\tilde R_1, \dots , \tilde R_k$ such that the following two properties hold: first, system~\eqref{eq:bl} admits a solution if and only if
\begin{equation}
 \label{eq:stable}
       ( \U_l - \bar \U ) \in \text{span} \langle \tilde  R_1, \dots, \tilde R_k \rangle.
\end{equation}
Second, the vectors $\tilde R_1, \dots , \tilde R_k, R_{k+1}, \dots , R_m$ constitute a basis of $\R^m$. 
Specific examples with explicit constructions of the vectors $\tilde R_1, \dots , \tilde R_k$  are 
discussed later. 

Consider the linear system 
\begin{equation}
 \label{eq:solverbrie}
         \U_l + \sum_{i=1}^k \alpha_i \tilde R_i +  
\sum_{i=k+1}^m \alpha_i R_i = \U_0,
\end{equation}
which by the second property of the vectors $\tilde R_1, \dots , \tilde R_k$ 
admits a unique solution $(\alpha_1, \dots, \alpha_m)$. The solution $\U$ obtained by taking the limit $\epsilon \to 0^+$
of~\eqref{eq:linv},\eqref{eq:linbrie} is then 
\begin{equation}
\label{eq:briesolver}
  \U (t, x) =
  \left\{
  \begin{array}{lll}
         \displaystyle U_b + \sum_{i=1}^k \alpha_i \tilde R_i & \text{if $0< x < \lambda_k t$} & \\
         \displaystyle U_b + \sum_{i=1}^k \alpha_i \tilde R_i + \sum_{i=k+1}^j \alpha_i R_i  
          &
         \text{if $\lambda_j t < x < \lambda_{j+1} t$}, & j =k+1, \dots, m-1  \\
         U_0   & \text{if $x > \lambda_m t$} &, \\
  \end{array} 
  \right.
\end{equation}  
where as usual $\lambda_1, \dots, \lambda_m$ denote the eigenvalues of the matrix $\A$. 
Note that this construction also works in the case when the matrix $\B$ is the identity provided that we set 
\begin{equation}
\label{eq:tilder}
\tilde R_i \dot{=} R_i \qquad \forall \, i=1, \dots, k.
\end{equation} 
\subsection{Explicit computations for the linearized shallow water equations}
\label{s:lish}
The above constructions are fairly general and abstract. We illustrate them by an example, i.e. the linearized shallow water equations of fluid flow (see LeVeque~\cite{LEV1}):
\begin{equation}
\label{eq:lsws}
\begin{aligned}
h_t + \widetilde{u} h_x + \widetilde{h} u_x &= 0, \\
u_t + g h_x + \widetilde{u} u_x &= 0.
\end{aligned}
\end{equation} 
Here, the height is denoted by $h$ and water velocity by $u$. The constant $g$ stands for the acceleration due to gravity and  $\widetilde{h},\widetilde{u}$ are the (constant) height and velocity states around which the shallow water
equations are linearized.

The physically relevant viscosity mechanism for the shallow water system is the \emph{eddy viscosity}. Adding eddy viscosity to the linearized shallow water system results in the following mixed hyperbolic-parabolic system:
\begin{equation}
\label{eq:lswsed}
\begin{aligned}
h_t + \widetilde{u} h_x + \widetilde{h} u_x &= 0,\\
u_t + g h_x + \widetilde{u} u_x &= \epsilon u_{xx}.
\end{aligned}
\end{equation}

For the sake of comparison, we add an \emph{artificial} viscosity to the linearized shallow waters by including the Laplacian. The resulting parabolic system is
\begin{equation}
\label{eq:lswslap}
\begin{aligned}
h_t + \widetilde{u} h_x + \widetilde{h} u_x &= \epsilon h_{xx}, \\
u_t + g h_x + \widetilde{u} u_x &= \epsilon u_{xx}.
\end{aligned}
\end{equation}

Systems~\eqref{eq:lswslap} and~\eqref{eq:lswsed} can be written in the form~\eqref{eq:linv} provided that
\begin{equation}
\label{eq:viscsw}
\A = \left(
\begin{array}{cc}
       \tilde u & \tilde h \\
       g & \tilde u  
\end{array}
\right), \quad 
\B =\B^{Lap} = \left(
\begin{array}{cc}
       1 & 0 \\
       0 & 1  
\end{array}
\right), \qquad 
\B= \B^{EDvisc}= \left(
\begin{array}{cc}
       0 & 0 \\
       0 & 1  
\end{array}
\right)
\end{equation}
in ~\eqref{eq:lswslap} and~\eqref{eq:lswsed}, respectively. We will construct explicit solutions for the linearized shallow water equations \eqref{eq:lsws} for the limit of both the eddy viscosity as well as the artificial viscosity. For the rest of this section, we specify the parameters
\begin{equation}
\label{eq:lsw:wtl}
\widetilde{h} = 2,\quad \widetilde{u} = 1, \quad g = 1.
\end{equation}
and consider the initial data
\begin{equation}
\label{eq:lswinit}
(h,u) (x, 0) = \begin{cases}
                     U^- = (3,1), &{\rm if} \quad x < 0, \\
                     U^+= (1,1), &{\rm if} \quad x > 0.
                     \end{cases}
\end{equation}
and the Dirichlet boundary data 
\begin{equation}
\label{eq:ldir}
   (h, u) (-1, t) = U_l (t) =(2,  1) \quad \forall \, t > 0.
\end{equation}

\subsubsection{Solution of the Riemann problem}
\label{s:riepnb}
We now apply the construction described in Section~\ref{s:linrie} to solve the Riemann problem~\eqref{eq:lsws},~\eqref{eq:lswinit}. 
The eigenvalues of the matrix $\A$ in~\eqref{eq:viscsw} are $\lambda_1 = 1 - \sqrt{2} < 0$ and $\lambda_2 = 1 + \sqrt{2} > 0$, 
with corresponding eigenvectors 
\begin{equation}
 \label{eq:r1} 
   R_1 = 
   \left(
\begin{array}{cc}
       1  \\
       - \sqrt{2}/ 2  \\  
\end{array}
\right) \qquad 
  R_2 = \left(
\begin{array}{cc}
       1  \\
       \sqrt{2}/ 2.  \\  
\end{array}
\right).
\end{equation}
Hence, the solution of the linear system~\eqref{eq:linsyst} in this case is $\alpha_1= \alpha_2= -1$.
\subsubsection{Solution of the boundary Riemann problem limit of the uniform (Laplacian) viscosity}
\label{s:lswe}
We now apply the construction in Section~\ref{s:blinrie} to determine the limit $\epsilon \to 0^+$ of the viscous 
approximation~\eqref{eq:lswslap} coupled with the Cauchy
and Dirichlet data~\eqref{eq:lswinit} and~\eqref{eq:ldir}. 

We solve the linear system~\eqref{eq:solverbrie} in the case when $\U_l$ is given by~\eqref{eq:ldir}, $U_0 = (3, 1)$ 
and~\eqref{eq:tilder} holds and we obtain $\alpha_1 = \alpha_2 = 1/2$.

By combining this with the analysis in Section~\ref{s:riepnb} and by recalling~\eqref{eq:riesolver} and~\eqref{eq:briesolver}, 
we conclude 
that the local in time solution
of~\eqref{eq:lsws} obtained by taking the limit $\epsilon \to 0^+$ of~\eqref{eq:lswslap},~\eqref{eq:lswinit},~\eqref{eq:ldir} is   
\begin{equation}
 \label{eq:lswriep}
(h, u) (t, x) =
  \left\{
  \begin{array}{lll}
          (5/2, 1- \sqrt{2}/4) & \text{if $0 < x + 1 < (1 + \sqrt{2} ) t$} \\
         (3, 1) & \text{if $x + 1 > (1 + \sqrt{2} ) t$ and $x < (1 - \sqrt{2} ) t$} \\
         (2, 1 + \sqrt{2}/2)  & \text{if $ (1 - \sqrt{2} ) t  < x < (1 + \sqrt{2} ) t$} \\
         (1, 1)  & \text{if $x > (1 + \sqrt{2} ) t$} \\
  \end{array}
  \right. 
\end{equation}
\subsubsection{Solution of the boundary Riemann problem limit of the eddy viscosity}
\label{s:lev}
We evaluate the limit $\varepsilon \to 0^+$ of the viscous approximation~\eqref{eq:lswsed} coupled with the Cauchy
and Dirichlet data~\eqref{eq:lswinit} and~\eqref{eq:ldir} by applying the construction described in Section~\ref{s:blinrie}. 
We consider system~\eqref{eq:bl} in the case when $\B$ is the same matrix 
$\B^{EDvisc}$ as in~\eqref{eq:viscsw} and $U_l$ is given by~\eqref{eq:ldir}
and we get 
\begin{equation}
\label{eq:lswev}
\left\{
\begin{array}{ll}
          0 = h - \bar h + 2 (u - \bar u ) \\
          \dot{u} = - (u - \bar u) \\
           (h, u) (0) = (2, 1) \qquad \lim_{y \to + \infty} (h, u) (y) = (\bar h, \bar u).
 \end{array}
\right.
\end{equation}
By imposing the initial datum $h (0)=2$, $u(0)=1$ one gets that~\eqref{eq:lswev} admits a solution if and only if 
$$
\left(
  \begin{array}{cc}
         2- \bar h \\
         1- \bar u \\
  \end{array}
\right) =
\text{span} \langle \tilde R_1 \rangle  \qquad \tilde R_1 \dot{=} \left(
  \begin{array}{cc}
       1,
       -1/2 \\
  \end{array}
\right)
$$
and hence by solving the linear system~\eqref{eq:solverbrie} we get in this case $\alpha_1 = \sqrt{2} / (\sqrt{2} +1)$
and $\alpha_2 = 1 / (\sqrt{2} +1)$. By combining this with the analysis in Section~\ref{s:riepnb} and by recalling~\eqref{eq:riesolver} and~\eqref{eq:briesolver} 
we conclude 
that the local in time solution
of~\eqref{eq:lsws} obtained by taking the limit $\epsilon \to 0^+$ of~\eqref{eq:lswsed},~\eqref{eq:lswinit},~\eqref{eq:ldir} is   
\begin{equation}
\label{eq:lswsolev}
(h, u) (t, x) =
  \left\{
  \begin{array}{lll}
       \displaystyle
         \left(
          (3 \sqrt{2}+2)/(\sqrt{2}+1), 
          (\sqrt{2}+2)/(2 \sqrt{2}+2) 
          \right) & \text{if $0 < x + 1 < (1 + \sqrt{2} ) t$} \\
         (3, 1) & \text{if $x + 1 > (1 + \sqrt{2} ) t$ and $x < (1 - \sqrt{2} ) t$}
 \\
         (2, 1 + \sqrt{2}/2)  & \text{if $ (1 - \sqrt{2} ) t  < x < (1 + \sqrt{2} ) t$} 
\\
         (1, 1)  & \text{if $x > (1 + \sqrt{2} ) t$} 
\\
  \end{array}
  \right. 
\end{equation}
The explicit calculations clearly show that the solution of the linearized shallow water equations \eqref{eq:lsws} realized as a limit of vanishing eddy viscosity \eqref{eq:lswsed} \emph{differs} from the solution realized as a limit of the artificial viscosity \eqref{eq:lswslap}. In particular, the solutions are different near the boundary at $x=0$ whereas they are the same, away from the boundary. The height ($h$) for both solutions is shown in figure \ref{fig:1}, left. 

\subsection{The solution of the boundary Riemann problem for non-linear systems.} 
\label{s:briep}
Consider the boundary Riemann problem obtained by coupling the mixed hyperbolic-parabolic system \eqref{eq:vcl}
with the Cauchy and Dirichlet data $ \U^{\epsilon}(0, x) =\U_0$, $ \U^{\epsilon}(t, 0)= \U_l$, respectively, 
 where $\U_0, \U_l \in \R^m$, and by then taking the limit $\epsilon \to 0^+$. 
One of the main challenges posed by this problem is establishing the relation between the data $\U_l$ and 
the trace 
$$
  \bar \U \dot{=} \lim_{x \to 0^+} \U(t, x). 
$$
As pointed out by Gisclon and Serre~\cite{G1,GS1}, there exists a \emph{boundary layer} ${\W: [0, +  \infty[ \to \R^m}$ such that
\begin{equation}
\label{eq:blnl}
\left\{
\begin{array}{lll}
 \B (W) \dot{W} = \F(\W) - \F(\bar U) \\
\W(0) = \U_l \qquad \lim_{y \to + \infty } \W(y) = \bar \U.
\end{array}
\right.
\end{equation}
In the case when the Jacobian matrix $\F_{\U}$ satisfies~\eqref{eq:ncharb} and the matrix $\B$ is invertible, the existence of 
a boundary layer $\W$ satisfying~\eqref{eq:blnl} is equivalent to the fact that $\U_l$ belongs to a 
suitable stable manifold centered at $\bar \U$. The general case when condition~\eqref{eq:ncharb} is violated (i.e., one eigenvalue of the Jacobian matrix $\F_\U$ can attain the value $0$) or the viscosity 
matrix $\B$ is singular is more complicated, but it can be treated under suitable assumptions.
A characterization of the solution of the boundary Riemann problem obtained as limit of 
$\varepsilon \to 0^+$ of the viscous approximation~\eqref{eq:vcl} is provided in Bianchini and 
Spinolo~\cite{BIASPI:ARMA} under the assumption that the Dirichlet and boundary data are sufficiently close. 
We also refer to Joseph and LeFloch~\cite{JOLEF:ARMA}, Ancona and Bianchini~\cite{ANBIA} and to the references therein for characterizations 
of solution of the boundary Riemann problem obtained as limits of different viscous approximations. We observe that in general it is difficult to compute the solution of the boundary Riemann problem for a non-linear system explicitly.

\section{Numerical schemes.}
\label{sec:schemes}
\subsection{Definition of the schemes}
We write down a semi-discrete conservative finite difference (finite volume) scheme for the system of conservation laws \eqref{eq:cl} as
\begin{equation}
\label{eq:fds}
\frac{d}{dt} \U_j(t) + \frac{1}{\Dx}\left(\F_{j+1/2} - \F_{j-1/2}\right) = 0.
\end{equation}
Here $\U_j \approx \U(x_j)$  with $x_j$ being the midpoint of the cell $[x_{j-1/2},x_{j+1/2}]$.  As stated in the introduction, the numerical flux $\F_{j+1/2} = \F(\U_j,\U_{j+1})$ is determined by (approximate) solutions of the Riemann problem at the interface $x_{j+1/2}$. An equivalent expression of the numerical flux (see the book by LeVeque~\cite{LEV1}) is 
\begin{equation}
\label{eq:nf1}
\F_{j+1/2} = \frac{\F(\U_j) + \F(\U_{j+1})}{2} -\frac{1}{2}\widehat{\D}_{j+1/2},
\end{equation}
with $\widehat{\D} = \widehat{\D}(\U_j,\U_{j+1})$ being the corresponding \emph{numerical diffusion} operator. As an example, the Roe diffusion operator (see again~\cite{LEV1}) is given by
\begin{equation}
\label{eq:roe}
\widehat{\D}_{j+1/2} = R_{j+1/2} |\Lambda_{j+1/2}| R^{-1}_{j+1/2} \jmp{\U}_{j+1/2},
\end{equation}
where $\Lambda$ and $R$ are matrix of eigenvalues and eigenvectors of the Jacobian $\F_{\U}$ evaluated at a suitable average state. Also, here and in the following we use the notations
$$
\overline{a}_{j+1/2} = \frac{a_j + a_{j+1}}{2}, \quad \jmp{a}_{j+1/2} = a_{j+1} - a_j.
$$

Clearly, the above numerical diffusion operator does not incorporate any information about the underlying viscous approximation. 
Hence, the approximate solutions generated by schemes such as the Roe scheme may not converge to the physically relevant 
solution, given as a limit of the underlying viscous approximation as the viscosity parameter goes to zero. An illustration is provided in figure \ref{fig:1}, right. Here, we present results obtained by approximating the linearized shallow water system \eqref{eq:lsws} with the Roe (Godunov) scheme. The exact solution, computed in \eqref{eq:lswsolev} as the limit of vanishing eddy viscosity \eqref{eq:lswsed} is also shown on the left. As shown in the figure, the Roe scheme does not converge to this physically relevant solution as the numerical viscosity \eqref{eq:roe} is very different from the eddy viscosity in \eqref{eq:lswsed}.

In this paper, we consider a different paradigm for numerically approximating the initial-boundary value problem. The main difference from the standard schemes lies in the choice of the numerical flux $\F$ in \eqref{eq:fds}. It combines the following ingredients.
\subsubsection{Entropy conservative fluxes:} We assume that~\eqref{eq:cl} admits an entropy-entropy flux pair $(S, Q)$ and by following Tadmor~\cite{TAD1} we define an entropy conservative flux for \eqref{eq:cl} as 
\begin{definition}
A numerical flux $\F^{\ast}_{j+1/2} = \F^{\ast}(\U_j,\U_{j+1})$ 
is defined to be entropy conservative for entropy $S$ if it satisfies 
\begin{equation}
\label{eq:ecf}
\jmp{\V}^{\top}_{j+1/2} \F^{\ast}_{j+1/2} = \jmp{\Psi}_{j+1/2}
\end{equation}
for every $j$. Here, $\V = S_{\U}$ is the vector of entropy variables and $\Psi = \V^{\top} \F - Q$ is the entropy potential for the entropy function $S$ and entropy flux  $Q$. 
\end{definition}
The existence of entropy conservative fluxes for system of conservation laws is shown in Tadmor~\cite{TAD1} and explicit examples of entropy conservative fluxes are summarized in Fjordholm, Mishra and Tadmor~\cite{FMT4}. 

{\color{black} It is known (see Dafermos~\cite[Chapter 7]{DAF1}) that if $(S, Q)$ is an entropy-entropy flux pair $(S, Q)$ with $S$ strictly convex,  then $\F_U \U_V$ is symmetric and $\U_\V$ is symmetric and positive definite. In the following, we also assume that the entropy is dissipative, namely that
\begin{equation}
\label{eq:dissen}
         \langle \B \U_\V\xi, \xi \rangle \ge 0 \quad \forall \;  \xi \in \R^m. 
\end{equation} 
This condition is satisfied in physical cases. In particular, it is satisfied in all the cases we discuss in the following.}
\subsubsection{Numerical diffusion operator:}
Let \eqref{eq:vcl} be the underlying mixed hyperbolic-parabolic regularization of the hyperbolic equation \eqref{eq:cl}. We choose a \emph{numerical diffusion} operator,
\begin{equation}
\label{eq:ndiff}
\D^{\ast}_{j+1/2} := \D^{\ast}(\U_j,\U_{j+1}) = c_{\max} \B(\hat{\U}_{j+1/2})\jmp{\U}_{j+1/2}.
\end{equation}
Here, $\B$ is the viscosity matrix in the parabolic regularization \eqref{eq:vcl} evaluated at some suitable averaged state $\hat{U}_{j+1/2}$ and 
\begin{equation}
\label{eq:maxe}
c_{\max} (t) = \max\limits_{j} |\lambda^{\max}_j|,
\end{equation}
with $\lambda_j^{\max}$ being the largest eigenvalue of the Jacobian $\F_\U$ at a given state $\U_j$. 
\subsubsection{Correct numerical diffusion (CND) scheme.}
We choose the numerical flux
\begin{equation}
\label{eq:nf2}
\F_{j+1/2} = \F^{\ast}_{j+1/2} - \frac{1}{2} \D^{\ast}_{j+1/2}.
\end{equation}
Here, $\F^{\ast}_{j+1/2} = \F^{\ast}\left(\U_j,\U_{j+1}\right)$ is an entropy conservative flux \eqref{eq:ecf} for the system \eqref{eq:cl} and the numerical diffusion operator $\D^{\ast}_{j+1/2}$ is defined in \eqref{eq:ndiff}. The semi-discrete scheme \eqref{eq:fds} with numerical flux \eqref{eq:nf2} has the following properties:
\begin{theorem}
\label{theo:1}
Assume that the system \eqref{eq:cl} is equipped with the entropy-entropy flux pair $(S, Q)$ which is {\color{black}dissipative in the sense of~\eqref{eq:dissen}.} Then, the scheme \eqref{eq:fds} with numerical flux \eqref{eq:nf2} satisfies
\begin{itemize}
\item [(i.)] a (local) discrete entropy inequality (discrete version of the entropy inequality \eqref{eq:ent}) of the form
\begin{equation}
\label{eq:de1}
\frac{d}{dt} S(\U_j) (t) + \frac{1}{\Dx}\left(\hat{Q}_{j+1/2} - \hat{Q}_{j+1/2}\right) \leq 0,
\end{equation}
with a numerical entropy flux $\hat{Q}$ that is consistent with the entropy flux $Q$.
\item [(ii.)] The scheme is first-order accurate and the equivalent equation is 
\begin{equation}
\label{eq:ee1}
\U^{\Dx}_t + \F(\U^{\Dx})_x = \frac{c_{\max}\Dx}{2}\left(\B(\U^{\Dx}) \U^{\Dx}_x\right)_x + \mathcal{O}(\Dx^2).
\end{equation}
\end{itemize}
\end{theorem}
\begin{proof}
Multiplying both sides of the scheme \eqref{eq:fds} by the entropy variable $\V_j^\top$, we obtain 
\begin{equation}
\label{eq:p1}
\begin{aligned}
\frac{d}{dt} S(\U_j) &+ \frac{1}{\Dx}\left(\widetilde{Q}_{j+1/2} - \widetilde{Q}_{j-1/2} \right) \\
&\quad = + \underbrace{\frac{1}{2\Dx}\left(\jmp{\V}^{\top}_{j+1/2}\F^{\ast}_{j+1/2} + \jmp{\V}^{\top}_{j-1/2}\F^{\ast}_{j-1/2} \right)}_{T_1} \\
&\quad - \frac{c_{max}}{4\Dx}\left(\jmp{\V}^{\top}_{j+1/2}\B(\U(\hat{\V}_{j+1/2}))\U_{\V}(\hat{\V}_{j+1/2})\jmp{\V}_{j+1/2}\right) \\
&\quad - \frac{c_{max}}{4\Dx}\left(\jmp{\V}^{\top}_{j-1/2}\B(\U(\hat{\V}_{j-1/2}))\U_{\V}(\hat{\V}_{j-1/2})\jmp{\V}_{j-1/2}\right).
\end{aligned}
\end{equation}
Here, we have introduced the numerical flux
$$
\widetilde{Q}_{j+1/2}:= \overline{\V}^{\top}_{j}\F^{\ast}_{j+1/2} - \frac{c_{max}}{2}\left( \overline{\V}^{\top}_{j+1/2} \B(\U(\hat{\V}_{j+1/2}))\jmp{\U}_{j+1/2}\right),
$$
with $\widetilde{Q}_{j-1/2}$ defined analogously. Also, we have introduced the average $\hat{\V}_{j+1/2}$ satisfying
\begin{equation}
\label{eq:av10}
\jmp{\U}_{j+1/2} = \U_{\V}(\hat{\V}_{j+1/2})\jmp{\V}_{j+1/2}.
\end{equation}
By using the definition of the entropy conservative flux \eqref{eq:ecf}, the term $T_1$ in \eqref{eq:p1} can be simplified as
{\color{black}\begin{align*}
T_1 &= \frac{1}{2\Dx}\left(\jmp{\V}^{\top}_{j+1/2}\F^{\ast}_{j+1/2} + \jmp{\V}^{\top}_{j-1/2}\F^{\ast}_{j-1/2} \right) \\
&= \frac{1}{2\Dx}\left(\jmp{\Psi}_{j+1/2} + \jmp{\Psi}_{j-1/2}\right) \\
&=  \frac{1}{\Dx}\left(\overline{\Psi}_{j+1/2} - \overline{\Psi}_{j-1/2}\right)
\end{align*}}
Substituting the above expression of $T_1$ in \eqref{eq:p1} yields
\begin{equation}
\label{eq:p2}
\begin{aligned}
\frac{d}{dt} S(\U_j) &+ \frac{1}{\Dx}\left(\hat{Q}_{j+1/2} - \hat{Q}_{j- 1/2} \right) = \\
&\quad - \frac{c_{max}}{4\Dx}\left(\jmp{\V}^{\top}_{j+1/2}\B(\U(\hat{\V}_{j+1/2}))\U_{\V}(\hat{\V}_{j+1/2})\jmp{\V}_{j+1/2}\right) \\
&\quad - \frac{c_{max}}{4\Dx}\left(\jmp{\V}^{\top}_{j-1/2}\B(\U(\hat{\V}_{j-1/2}))\U_{\V}(\hat{\V}_{j-1/2})\jmp{\V}_{j-1/2}\right).
\end{aligned}
\end{equation}
with
$$
\hat{Q}_{j+1/2}:= \overline{\V}^{\top}_{j+1/2}\F^{\ast}_{j+1/2} - \overline{\Psi}_{j+1/2} - \frac{c_{max}}{2}\left( \overline{\V}^{\top}_{j+1/2} \B(\U(\hat{\V}_{j+1/2}))\jmp{\U}_{j+1/2}\right)
$$
and $\hat{Q}_{j- 1/2}$ defined analogously. 
Note that $\hat{Q}(a,a) = Q(a)$ from the definition of the entropy potential $\Psi$. The discrete entropy inequality \eqref{eq:de1} follows from assumption~\eqref{eq:dissen}.

The equivalent equation \eqref{eq:ee1} is a simple consequence of Taylor expansion and reveals the first order accuracy of the scheme.
\end{proof}
Thus, the proposed numerical scheme is entropy stable under reasonable hypotheses on the system \eqref{eq:cl}. Furthermore, the equivalent equation \eqref{eq:ee1} shows that the numerical viscosity of this scheme matches the underlying physical viscosity operator in \eqref{eq:vcl} at leading order. Hence, we claim that the scheme \eqref{eq:fds} with numerical flux \eqref{eq:nf2} incorporates the \emph{correct numerical dissipation} and term it as the \emph{CND} scheme. 

The Dirichlet boundary conditions for \eqref{eq:cl} are imposed weakly by setting
\begin{equation}
\label{eq:bcn}
\U_0(t) = \U_l(t).
\end{equation}
This amounts to setting the Dirichlet data as the value in the ghost cell $[x_{-1/2},x_{1/2}]$.

The semi-discrete scheme \eqref{eq:fds} is integrated in time with the SSP-RK2 time integrator:
\begin{equation}
\label{eq:SSP}
\begin{aligned}
\U^{\ast} &= \U^n + \Delta t{\mathcal L}(\U^n), \\
\U^{\ast \ast} &= \U^\ast + \Delta t{\mathcal L}(\U^{\ast}), \\
\U^{n+1} &= \frac{1}{2}(\U^n + \U^{\ast \ast}),
\end{aligned}
\end{equation}
that approximates the ODE system 
\begin{equation}
\label{eq:L}
\frac{d}{dt} \U(t) = {\mathcal L}(\U(t)),
\end{equation}
for the unknowns $\U = \{\U_j\}_j$, defined by the scheme \eqref{eq:fds}.
\subsection{Linear systems:}
We illustrate the finite difference scheme \eqref{eq:fds} for a linear system, i.e for~\eqref{eq:cl} (and the parabolic regularization \eqref{eq:vcl}) with 
\begin{equation}
\label{eq:linsys}
\F(\U) = \A\U, \quad \B(\U) = \B,
\end{equation}
for $\A,\B$ given $(m\times m)$-matrices. {\color{black}As pointed out before, if $S (\U)=\frac{1}{2}\U^T \mathcal S \U$ is a strictly convex entropy, then the matrix $\mathcal S \A$ is symmetric.  Following Fjordholm, Mishra and Tadmor~\cite{FMT4}), we define the corresponding entropy conservative flux as}
\begin{equation}
\label{eq:eclin}
\F^{\ast}_{j+1/2} = \A \overline{\U}_{j+1/2}.
\end{equation}
We consider the following specific example:
\subsubsection{Linearized shallow water system.}
The linearized shallow water system \eqref{eq:lsws} is considered. We assign the data~\eqref{eq:lswinit}
and~\eqref{eq:ldir}. The computational domain is $[-1,1]$ and we use open (Neumann type) boundary conditions at the right boundary $x=1$. 

The numerical solutions computed with the standard Roe scheme~\eqref{eq:fds} and the CND scheme~\eqref{eq:nf2} are shown in figure \ref{fig:2}. As we are interested in computing the physically relevant solutions of the linearized shallow water equations, obtained as a limit of the eddy viscosity \eqref{eq:lswsed}, we also plot the exact solution computed in \eqref{eq:lswsolev} for comparison. Both the numerical solutions are computed with a $1000$ mesh points.

The results in figure \ref{fig:2} clearly show that the Roe scheme does not converge to the physically relevant solution \eqref{eq:lswsolev}. On the other hand, the solutions computed with the CND scheme approximate the physically relevant solution \eqref{eq:lswsolev} quite well. There are some small amplitude oscillations in the height with the CND scheme. This is a consequence of the singularity of the viscosity matrix $\B$ in this case. The experiment clearly shows that incorporating explicit information about the underlying viscous mechanism in the numerical diffusion operator results in the approximation of the correct solution.

\begin{figure}[htbp]
\centering
\subfigure[Height ($h$)]{\includegraphics[width=9cm]{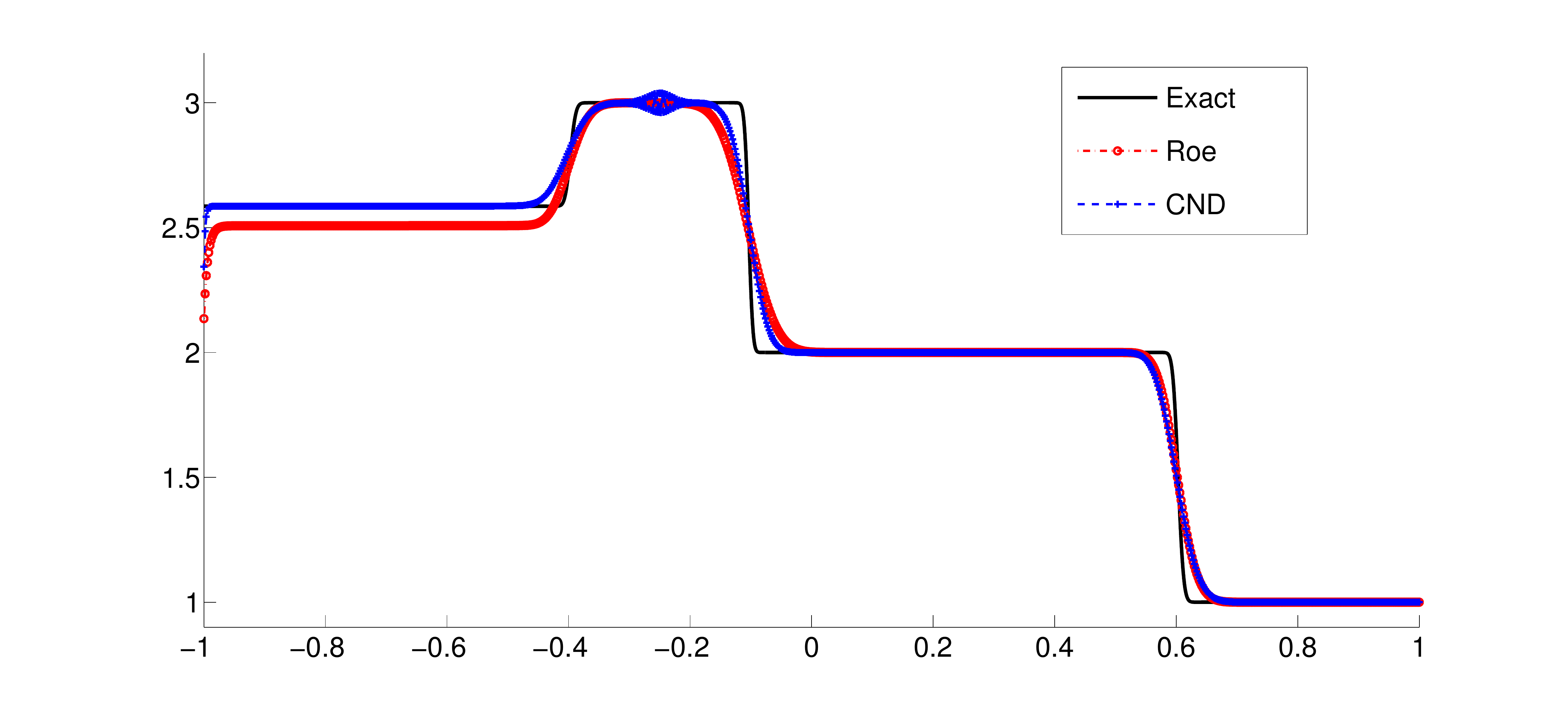}} \\
\subfigure[Velocity ($u$)]{\includegraphics[width=9cm]{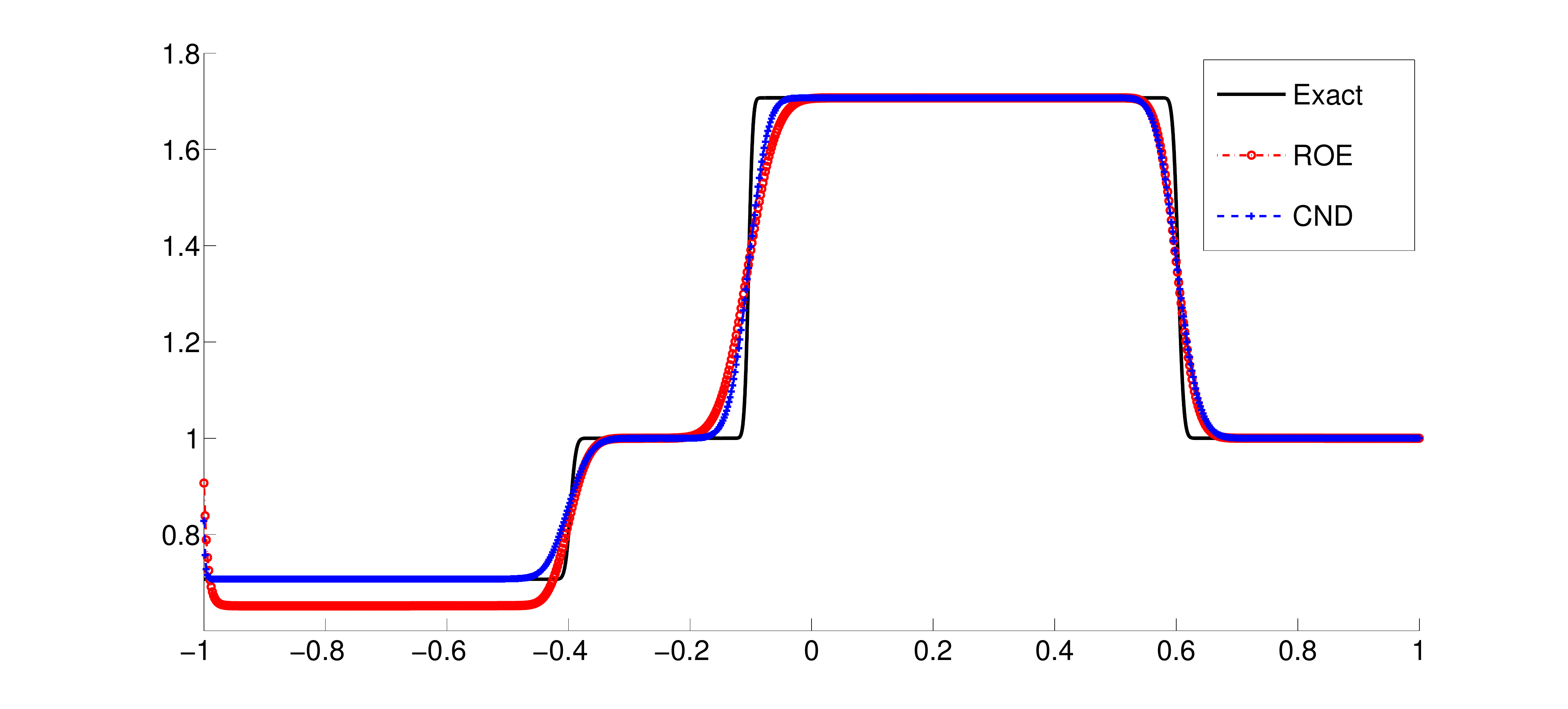}}
\caption{Solutions of the linearized shallow water equations \eqref{eq:lsws} with initial data \eqref{eq:lswinit} and boundary data \eqref{eq:ldir} computed with the Roe and CND schemes with $1000$ mesh points. The exact solution computed in \eqref{eq:lswsolev} is provided for comparison.}
\label{fig:2}
\end{figure}
\subsection{Nonlinear Euler equations}
In one space dimension, the Euler equations of gas dynamics are
\begin{equation}
\label{eq:euler}
\begin{aligned}
\rho_t + (\rho u)_x &= 0, \\
(\rho u)_t + (\rho u^2 + p)_x &= 0, \\
E_t   + ((E+p)u)_x &= 0.
\end{aligned}
\end{equation}
Here, $\rho$ is the fluid density and $u$ is the velocity. The total energy $E$ and the pressure $p$ are related by the ideal gas equation of state:
\begin{equation}
\label{eq:eos}
E = \frac{p}{\gamma -1 } + \frac{1}{2} \rho u^2,
\end{equation}
with $\gamma>1$ being a constant specific of the gas.

The system is hyperbolic with eigenvalues
\begin{equation}
\label{eq:euleig}
\lambda_1 = u - c, \quad \lambda_2 = u, \quad \lambda_3 = u +c.
\end{equation}
Here, $c = \sqrt{\frac{\gamma p}{\rho}}$ is the sound speed.

 The equations are augmented with the entropy inequality
 \begin{equation}
 \label{eq:eulent}
 \left(\frac{-\rho s}{\gamma - 1} \right)_t + \left(\frac{-\rho u s}{\gamma - 1}\right)_x \leq 0,
 \end{equation}
with thermodynamic entropy
$$
s = \log(p) - \gamma \log(\rho).
$$

The compressible Euler equations are derived by ignoring kinematic viscosity and heat conduction. Taking these small scale effects into account results in the \emph{compressible Navier-Stokes} equations:
\begin{equation}
\label{eq:cns}
\begin{aligned}
\rho_t + (\rho u)_x &= 0, \\
(\rho u)_t + (\rho u^2 + p)_x &= \nu u_{xx}, \\
E_t   + ((E+p)u)_x &= \nu \left(\frac{u^2}{2}\right)_{xx} + \kappa \theta_{xx}.
\end{aligned}
\end{equation}
Here, $\theta$ is the temperature given by
$$
\theta = \frac{p}{(\gamma-1)\rho}.
$$
The viscosity coefficient is denoted by $\nu$ and $\kappa$ is the coefficient of heat conduction.

For the sake of comparison, we add an uniform (Laplacian) diffusion to obtain the compressible Euler equations with \emph{artificial} viscosity:
\begin{equation}
\label{eq:eulerlap}
\begin{aligned}
\rho_t + (\rho u)_x &= \epsilon \rho_{xx}, \\
(\rho u)_t + (\rho u^2 + p)_x &= \epsilon (\rho u)_{xx}, \\
E_t   + ((E+p)u)_x &= \epsilon E_{xx}.
\end{aligned}
\end{equation}
To evaluate the limit solution of~\eqref{eq:cns}, we construct a numerical approximation by discretizing the mixed hyperbolic-parabolic systems \eqref{eq:cns} and \eqref{eq:eulerlap} for a fixed and very small value of the viscosity coefficient. We do so by the (semi-discrete) finite difference scheme
\begin{equation}
\label{eq:fdsv}
\frac{d}{dt} \U_j(t) + \frac{1}{\Dx}\left(\F^{\ast}_{j+1/2} - \F^{\ast}_{j-1/2}\right) = \frac{\epsilon}{\Dx^2} \D_j.
\end{equation}
Here, the numerical flux is the entropy conservative flux \eqref{eq:ecf} (see Ismail and Roe ~\cite{IR1})
\begin{equation}
\label{eq:eulerec}
\begin{aligned}
\F^{\ast}_{j+1/2}   &= [\F^{1,\ast}_{j+1/2}, \F^{2,\ast}_{j+1/2},\F^{3,\ast}_{j+1/2}]^{\top}, \\
\F^{1,\ast}_{j+1/2} &= (\overline{z_2})_{j+1/2} (z_3)^L_{j+1/2}, \quad 
\F^{2,\ast}_{j+1/2} = \frac{(\overline{z_3})_{j+1/2}}{(\overline{z_1})_{j+1/2}} + \F^{1,\ast}_{j+1/2}, \\
\F^{3,\ast}_{j+1/2} &= \frac{1}{2}\frac{(\overline{z_2})_{j+1/2}}{(\overline{z_1})_{j+1/2}}\left(\frac{\gamma +1}{\gamma -1} \frac{(z_3)^L_{j+1/2}}{(z_1)^L_{j+1/2}} + \F^{2,\ast}_{j+1/2} \right)
\end{aligned}
\end{equation}
with parameter vectors
\begin{equation}
\label{eq:param}
(z_1,z_2,z_3) = \left(\sqrt{\frac{\rho}{p}}, \sqrt{\frac{\rho}{p}} u, \sqrt{\rho p} \right).
\end{equation} 
The logarithmic mean of any quantity $a$, defined on the mesh, is denoted by
$$
(a)^L_{j+1/2} = \frac{a_{j+1} - a_j}{\log(a_{j+1}) - \log(a_j)}.
$$
We define the numerical diffusion operators by setting 
\begin{equation}
\label{eq:ev2}
\begin{aligned}
\D_j &= \left[\D^1_j,\D^2_j,\D^3_j\right]^{\top}, \\ 
\D^1_j & = 0, \\
\D^2_j  &= \nu(u_{j+1} - 2u_j + u_{j-1}), \\
\D^3_j &= \frac{\nu}{2}(u^2_{j+1} - 2u^2_j + u^2_{j-1}) + \kappa(\theta_{j+1} - 2\theta_j + \theta_{j-1}).
\end{aligned}
\end{equation}
and
\begin{equation}
\label{eq:ev1}
\D_j = \U_{j+1} - 2\U_j + \U_{j-1},
\end{equation}
for the compressible Navier-Stokes equations \eqref{eq:cns} and the Euler equations with artificial viscosity \eqref{eq:eulerlap}, respectively.

As an example, we consider both \eqref{eq:eulerlap} and \eqref{eq:cns} in the domain $[-1,1]$ with initial data
\begin{equation}
\label{eq:eulinit}
(\rho_0,u_0,p_0) = \begin{cases}
                     (3.0,1.0,3.0), &{\rm if} \; x < 0, \\
                     (1.0,1.0,1.0), &{\rm if} \; x > 0.
                     \end{cases}
\end{equation}
We impose open boundary conditions at the right boundary and Dirichlet boundary conditions at the left boundary with boundary data
\begin{equation}
\label{eq:eulbd}
\Big(\rho(-1,t),u(-1,t),p(-1,t) \Big)= (2.0,1.0,2.0).
\end{equation}
We set $\nu = \kappa = \epsilon$. The results for the finite difference scheme approximating the uniform viscosity  \eqref{eq:eulerlap} and the physical viscosity \eqref{eq:cns} are presented in Figure \ref{fig:3}. The figure shows that the there is a clear difference in the limit solutions of this problem, obtained from the compressible Navier-Stokes equations \eqref{eq:cns} and the Euler equations with artificial viscosity \eqref{eq:eulerlap}. The difference is more pronounced in the density variable near the left boundary. Both the limit solutions were computed by setting $\epsilon = 10^{-5}$ and on a very fine mesh of $32000$ points. 
\begin{figure}[htbp]
\centering
\subfigure[Density ($\rho$)]{\includegraphics[width=9cm]{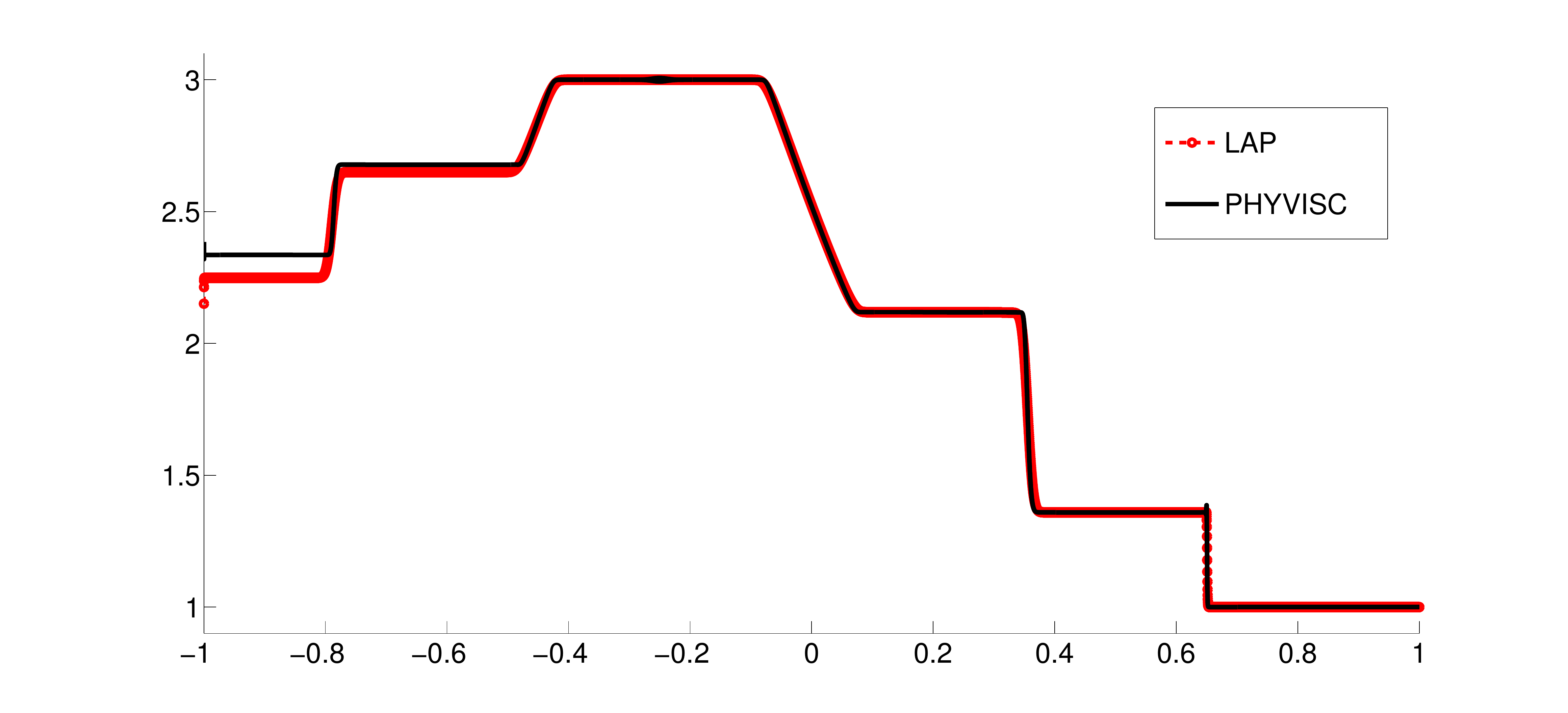}} \\
\subfigure[Velocity ($u$)]{\includegraphics[width=9cm]{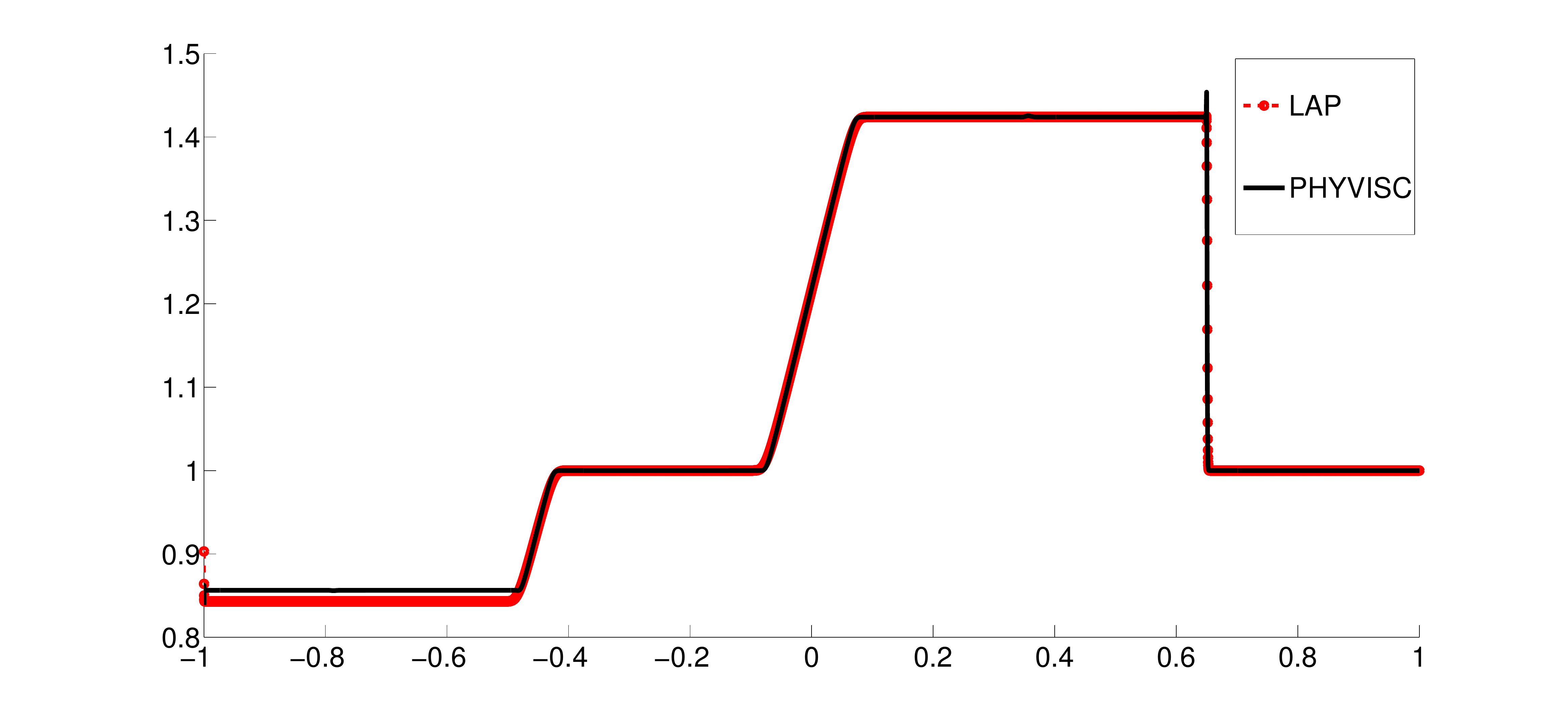}} \\
\subfigure[Pressure ($p$)]{\includegraphics[width=9cm]{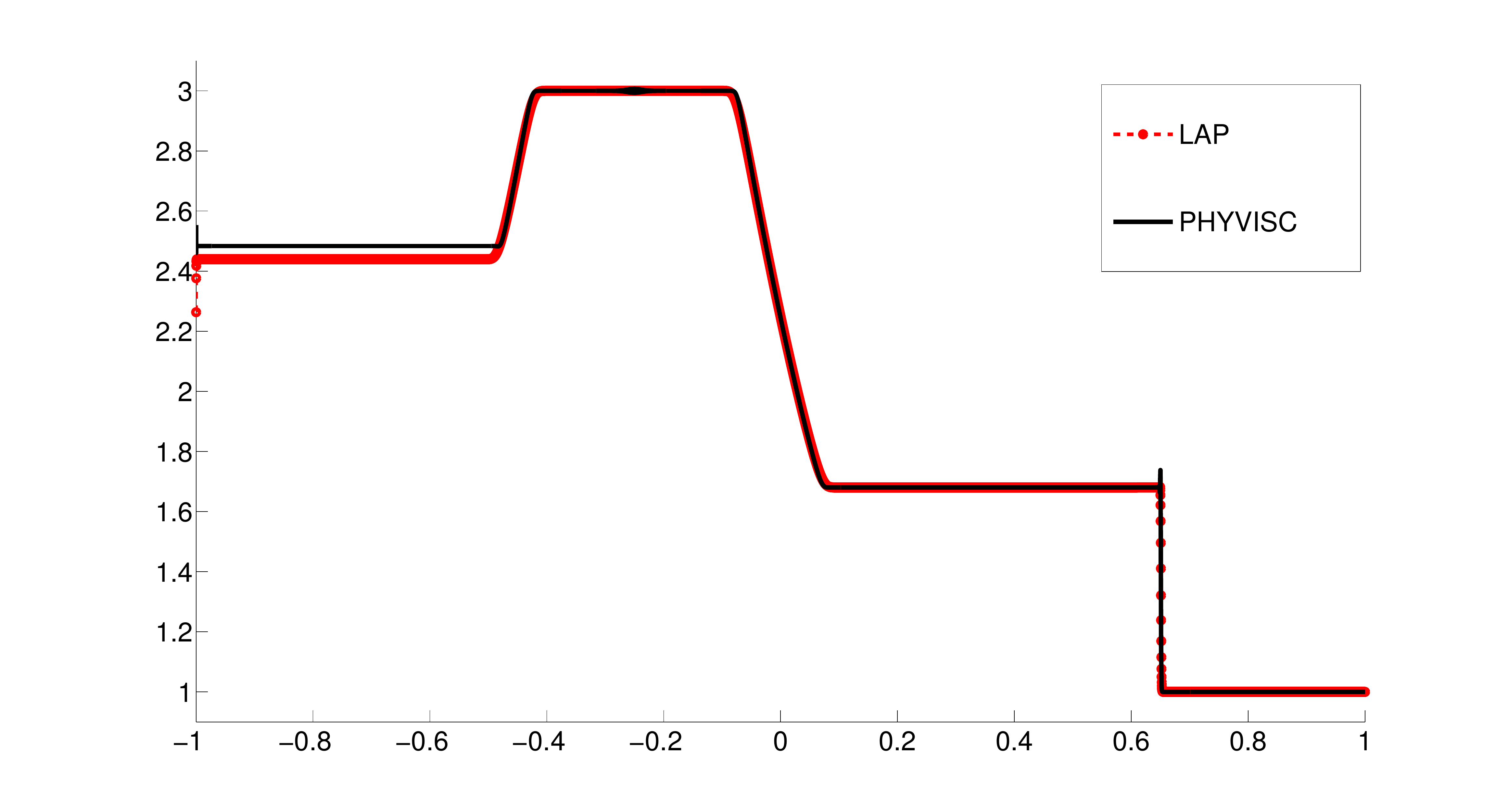}} 
\caption{Limit solutions of the compressible Euler equations \eqref{eq:euler} with initial data \eqref{eq:eulinit} and boundary data \eqref{eq:eulbd}. The limits of the physical viscosity i.e compressible Navier-Stokes equations \eqref{eq:cns} and the artificial (Laplacian) viscosity \eqref{eq:eulerlap} are compared.}
\label{fig:3}
\end{figure}
\begin{remark}
The above example also illustrates the limitations of using a mixed hyperbolic-parabolic system like the compressible Navier-Stokes equations \eqref{eq:cns}. In order to resolve the viscous scales, we need to choose $\Dx = \mathcal{O}\left(\frac{1}{\epsilon} \right)$, with $\epsilon$ being the viscosity parameter. As $\epsilon$ is very small in practice, the computational effort involved is prohibitively expensive. In the above example, we needed $32000$ points to handle $\epsilon = 10^{-5}$. Such ultra fine grids are not feasible, particularly in several space dimensions.
\end{remark}
\subsubsection{CND scheme for the Euler equations}
The CND scheme \eqref{eq:fds} for the Euler equations \eqref{eq:euler} is specified as follows: the entropy conservative flux in \eqref{eq:nf2} is given by \eqref{eq:eulerec} and the numerical diffusion operator in \eqref{eq:nf2} matches the kinematic viscosity and heat conduction of the compressible Navier-Stokes equations since it is defined by setting
\begin{equation}
\label{eq:eulernd}
\begin{aligned}
\D^{\ast}_{j+1/2} &= \left[\D^{1}_{j+1/2},\D^{2}_{j+1/2},\D^{3}_{j+1/2}\right]^{\top}, \\
\D^{1}_{j+1/2} &= 0, \\
\D^{2}_{j+1/2} &=   \left(\max\limits_{j} \left(|u_j| + \sqrt{\frac{\gamma p_j}{\rho_j}}\right)\right)\left(u_{j+1} - u_j\right), \\
D^{3}_{j+1/2} &=   \left(\max\limits_{j} \left(|u_j| + \sqrt{\frac{\gamma p_j}{\rho_j}}\right)\right)\left(\frac{1}{2}(u^2_{j+1} -u^2_j) + (\theta_{j+1} - \theta_j) \right).
\end{aligned}
\end{equation}
We discretize the initial-boundary value problem for the compressible Euler equations \eqref{eq:euler} on the computational domain $[-1,1]$ with initial data \eqref{eq:eulinit} and Dirichlet data \eqref{eq:eulbd}. The results with the CND scheme and a standard Roe scheme are shown in figure \ref{fig:4}. We present approximate solutions, computed on a mesh of $1000$ points, for both schemes. Both the Roe and the CND schemes have converged at this resolution. As we are interesting in approximating the physically relevant solutions of the Euler equations, realized as a limit of the Navier-Stokes equations, we plot a reference solution computed on a mesh of $32000$ points of the compressible Navier-Stokes equations \eqref{eq:cns} with $\kappa = \nu = 10^{-5}$. The figure shows that the Roe scheme clearly converges to an incorrect solution near the left boundary. This lack of convergence is most pronounced in the density variable. Similar results were also obtained with the standard Rusanov, HLL and HLLC solvers (see the book by LeVeque\cite{LEV1} for a detailed description of these solvers).

On the other hand, the CND scheme converges to the physically relevant solution. There are slight oscillations with the CND scheme as the numerical diffusion operator is singular. However, these oscillations do not impact on the convergence properties of this scheme. Furthermore, the CND scheme is slightly more accurate than the Roe scheme when both of them converge to the same solution (see near the interior contact).
\begin{figure}[htbp]
\centering
\subfigure[Density ($\rho$)]{\includegraphics[width=9cm]{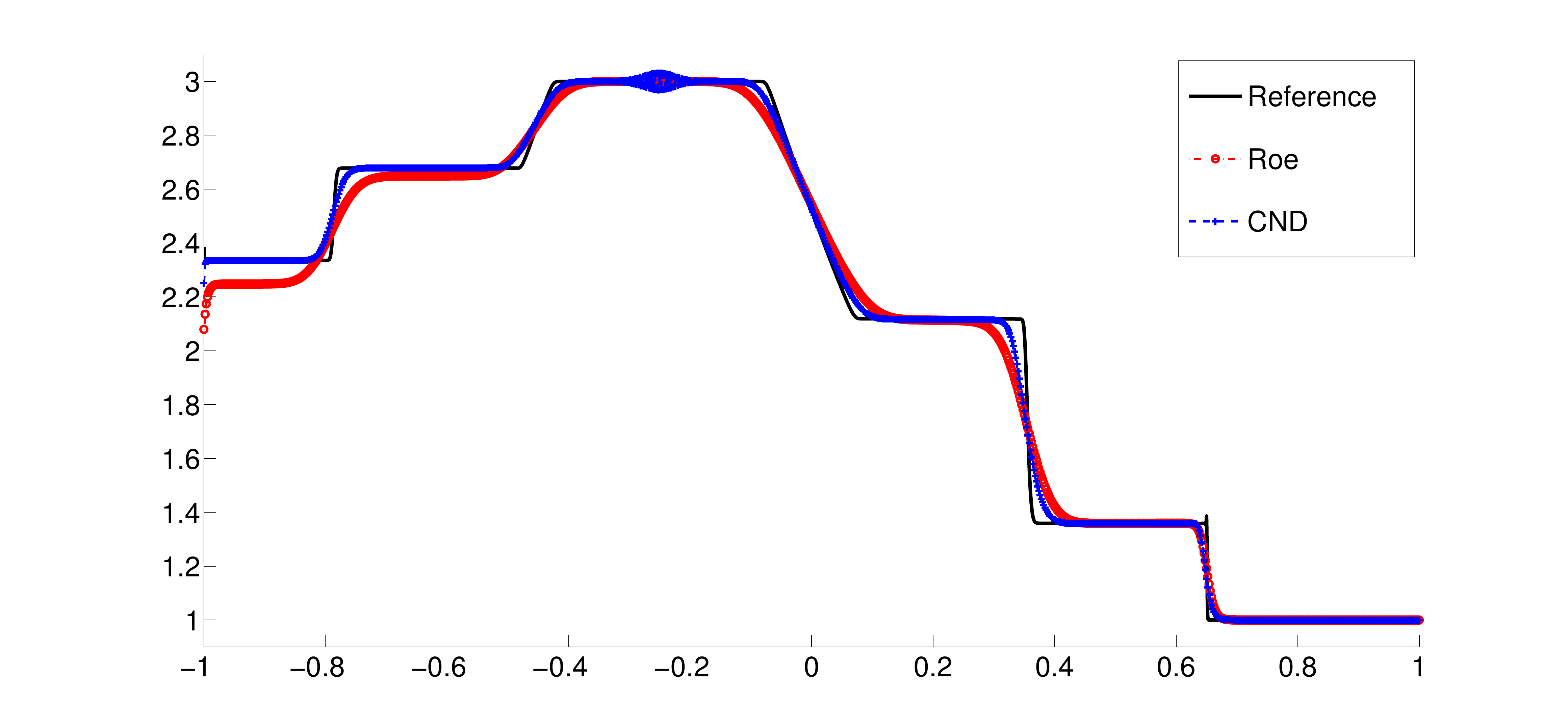}} \\
\subfigure[Velocity ($u$)]{\includegraphics[width=9cm]{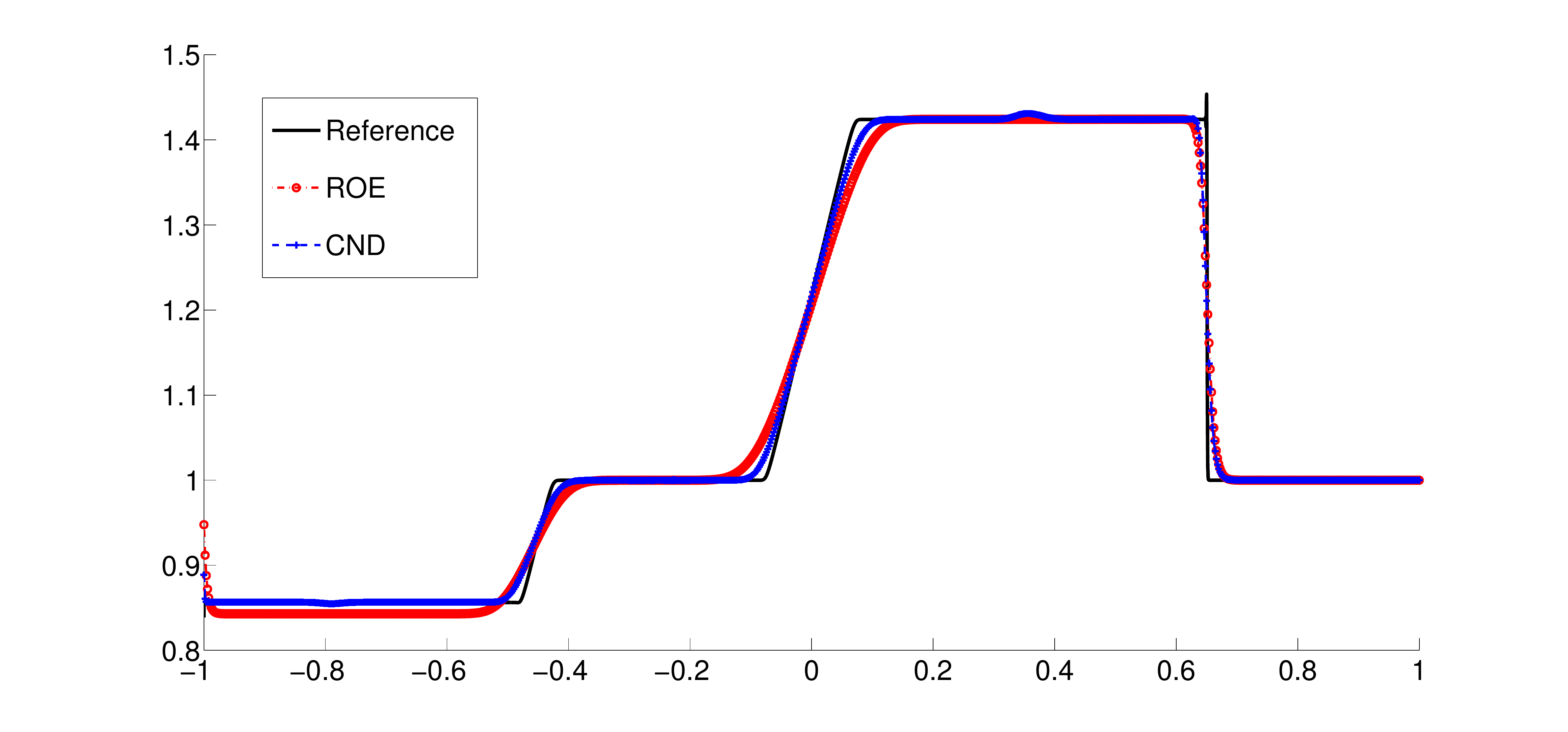}} \\
\subfigure[Pressure ($p$)]{\includegraphics[width=9cm]{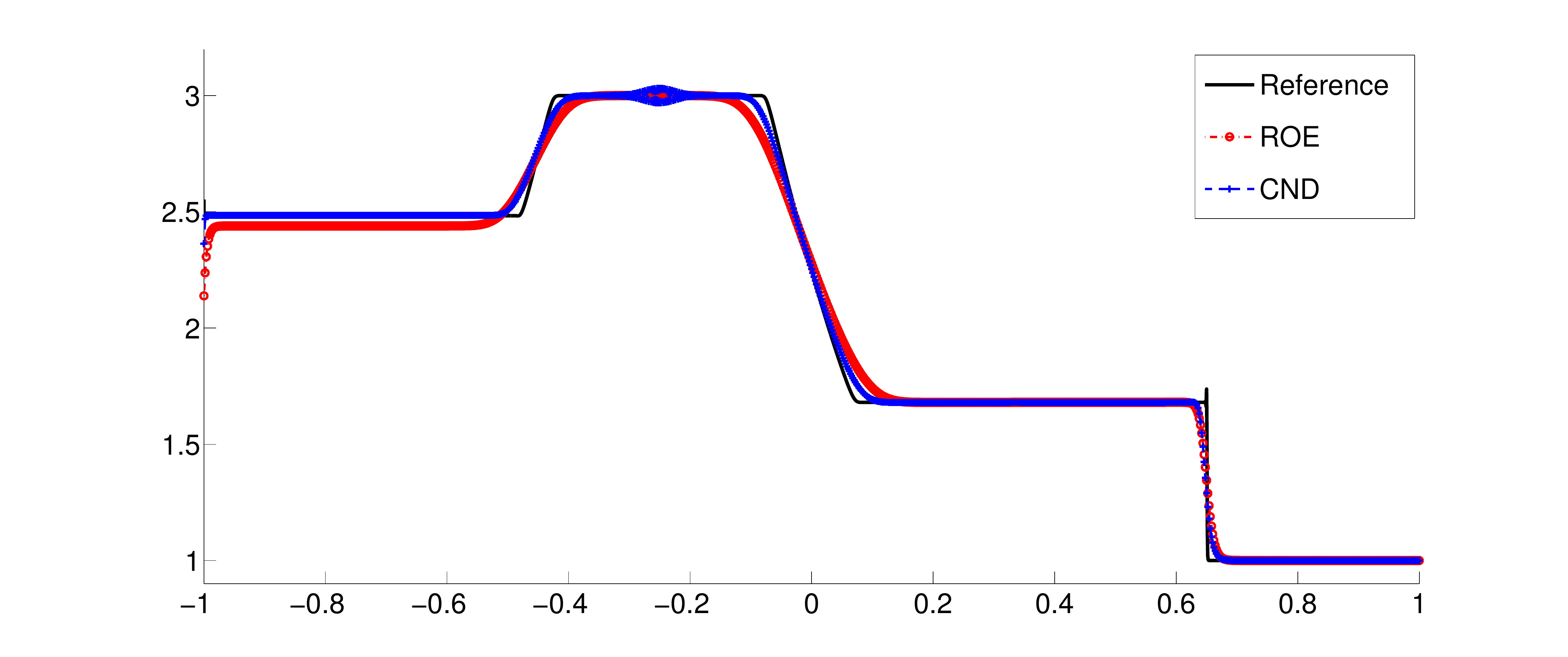}} 
\caption{Approximate solutions of the compressible Euler equations \eqref{eq:euler} with initial data \eqref{eq:eulinit} and boundary data \eqref{eq:eulbd}. We compare the Roe and CND schemes on $1000$ mesh points with a reference solution of the compressible Navier-Stokes equations \eqref{eq:cns} with $\kappa = \nu = 10^{-5}$.}
\label{fig:4}
\end{figure}
\section{Second-order CND schemes}
\label{sec:so}
The CND scheme, described in the last section, was first-order accurate in space. Consequently, it approximated shocks and contact discontinuities with excessive smearing, particularly on coarse meshes. We can improve the resolution of numerical schemes by constructing second-order accurate schemes. 

To this end, we reconstruct the cell averages $\U_j$ of the unknown to a piecewise linear function given by
\begin{equation}
\label{eq:2o1}
{\bf p}_{j}(x) := \U_{j} + \frac{\U^{\prime}_{j}}{\Delta x}(x - x_j).
\end{equation}
The numerical derivative $\U^{\prime}$ is chosen to be non-oscillatory by limiting the slope, i.e. setting 
\begin{equation}
\label{eq:2o2}
\U^{\prime}_{j} = {\rm minmod}(\U_{j+1} - \U_{j}, \U_{j} - \U_{j-1}),
\end{equation}
with the minmod function defined as 
\begin{equation}
{\rm minmod}(a,b)= \left\{
\begin{aligned}
 & sgn(a)\min\{|a|,|b|\}, &{\rm if} \quad sgn(a) = sgn(b), \\
& 0, \quad \quad {\rm otherwise} .
\end{aligned}
\right.
\end{equation}
Other limiters like the MC and Superbee limiters can also be chosen  (see the book by LeVeque\cite{LEV1} for the corresponding definitions). We need the cell interface values
\begin{equation}
\label{eq:2o4}
\U^{+}_{j} := {\bf p}_{j}(x_{j+1/2}),\quad \U^{-}_{j} := {\bf p}_{j}(x_{j-1/2}).
\end{equation}
With these reconstructed values, we modify the numerical flux \eqref{eq:nf2} by setting
\begin{equation}
\label{eq:sonf}
\F_{j+1/2} = \F^{\ast}_{j+1/2} - \frac{1}{2} \tilde{\D}_{j+1/2},
\end{equation}
with 
\begin{equation}
\label{eq:ndiff2}
\tilde{\D}_{j+1/2} = \tilde{\D}(\U^+_j,\U^-_{j+1}) = c_{\max} \B(\overline{\U}_{j+1/2})\left(\U^-_{j+1} - \U^+_j\right), 
\end{equation}
where the constant $c_{\max}$ is the same as in~\eqref{eq:maxe}.
Note that the \emph{only difference} between the flux \eqref{eq:nf2} and the flux \eqref{eq:sonf} lies in replacing the difference in cell averages in the numerical diffusion operator in \eqref{eq:ndiff} with the difference in the corresponding reconstructed edge values in \eqref{eq:ndiff2}. The overall scheme \eqref{eq:fds} with numerical flux \eqref{eq:sonf} is (formally) second-order accurate as the entropy conservative flux $\F^{\ast}$ is second order accurate (see Tadmor~\cite{TAD1}) and the difference in the numerical diffusion operator is a difference of second-order reconstructed values, see Fjordholm, Mishra and Tadmor~\cite{FMT4} for a proof of the order of accuracy of schemes constructed with numerical fluxes like \eqref{eq:sonf}. 

We test this second-order scheme \eqref{eq:fds}, \eqref{eq:sonf} for the compressible Euler equations. Let the computational domain be $[-1,1]$ with initial data \eqref{eq:eulinit} and Dirichlet data \eqref{eq:eulbd}. 

The scheme \eqref{eq:fds} is specified as follows: the entropy conservative flux in numerical flux \eqref{eq:sonf} is given by \eqref{eq:eulerec}. The numerical diffusion is 
\begin{equation}
\label{eq:eulernd2}
\begin{aligned}
\D^{\ast}_{j+1/2} &= \left[\D^{1}_{j+1/2},\D^{2}_{j+1/2},\D^{3}_{j+1/2}\right]^{\top}, \\
\D^{1}_{j+1/2} &= 0, \\
\D^{2}_{j+1/2} &=   \left(\max\limits_{j} \left(|u_j| + \sqrt{\frac{\gamma p_j}{\rho_j}}\right)\right)\left(u^-_{j+1} - u^+_j\right), \\
D^{3}_{j+1/2} &=   \left(\max\limits_{j} \left(|u_j| + \sqrt{\frac{\gamma p_j}{\rho_j}}\right)\right)\left(\frac{1}{2}((u^-_{j+1})^2 -(u^+_j)^2) + (\theta^-_{j+1} - \theta^+_j) \right),
\end{aligned}
\end{equation}
with $u^{\pm},\theta^{\pm}$ being obtained from the reconstructed conservative variables. The overall scheme (integrated in time with the SSP RK2 time stepping \eqref{eq:SSP}) is termed as the \emph{CND2} scheme.
\begin{figure}[htbp]
\centering
\subfigure[Density ($\rho$)]{\includegraphics[width=9cm]{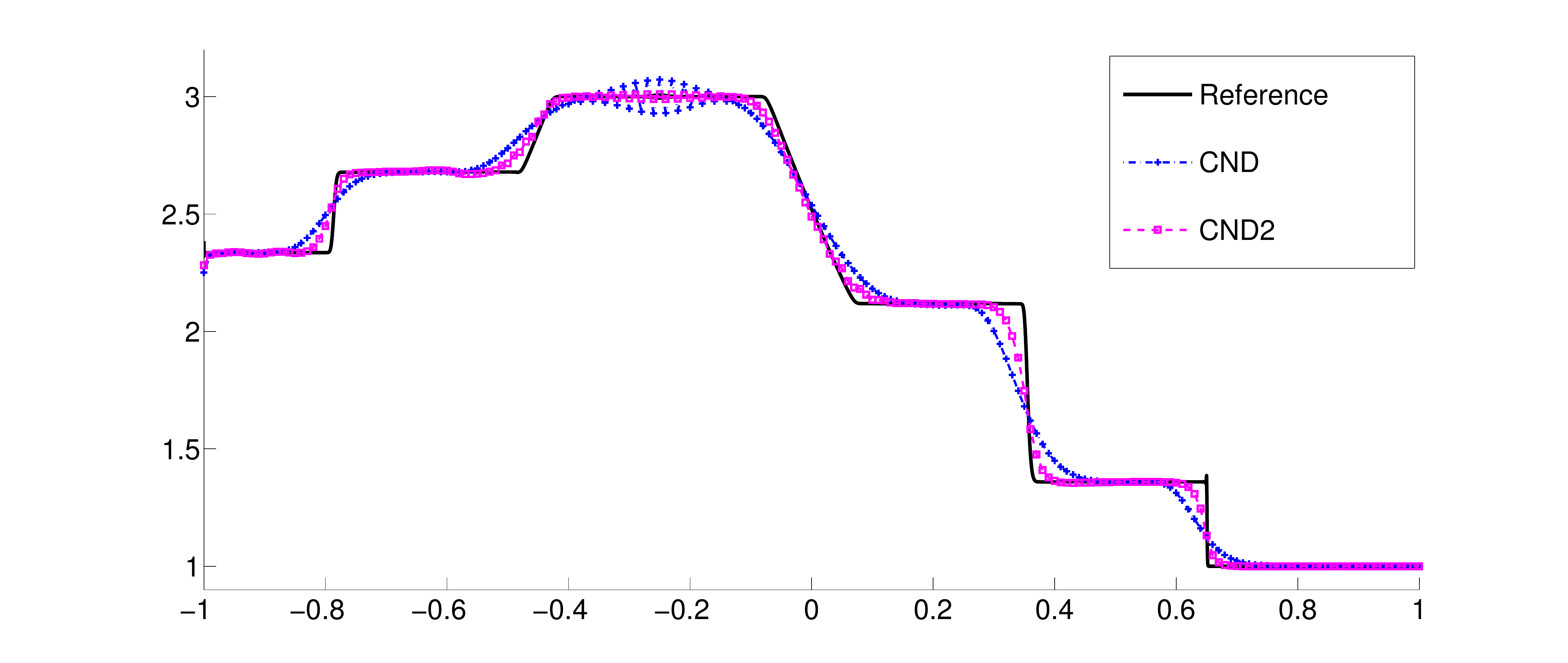}} \\
\subfigure[Velocity ($u$)]{\includegraphics[width=9cm]{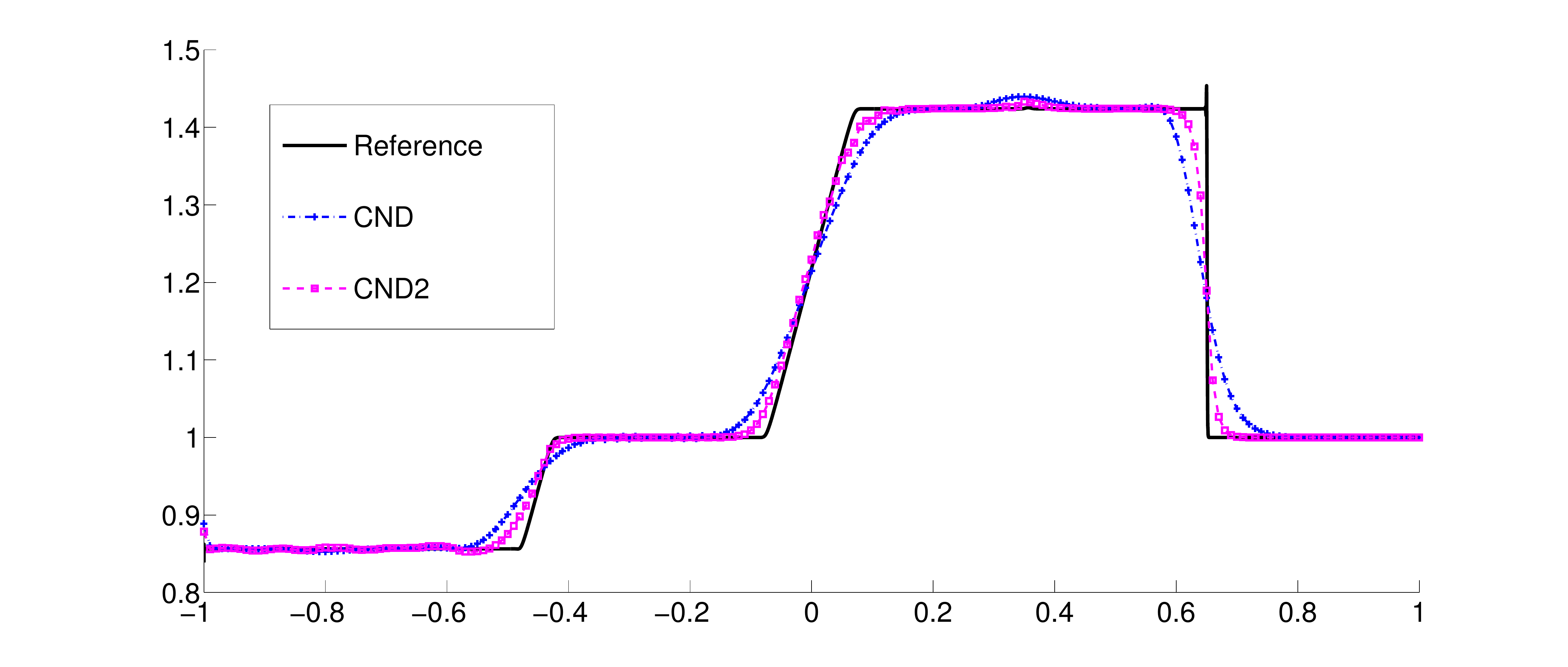}} \\
\subfigure[Pressure ($p$)]{\includegraphics[width=9cm]{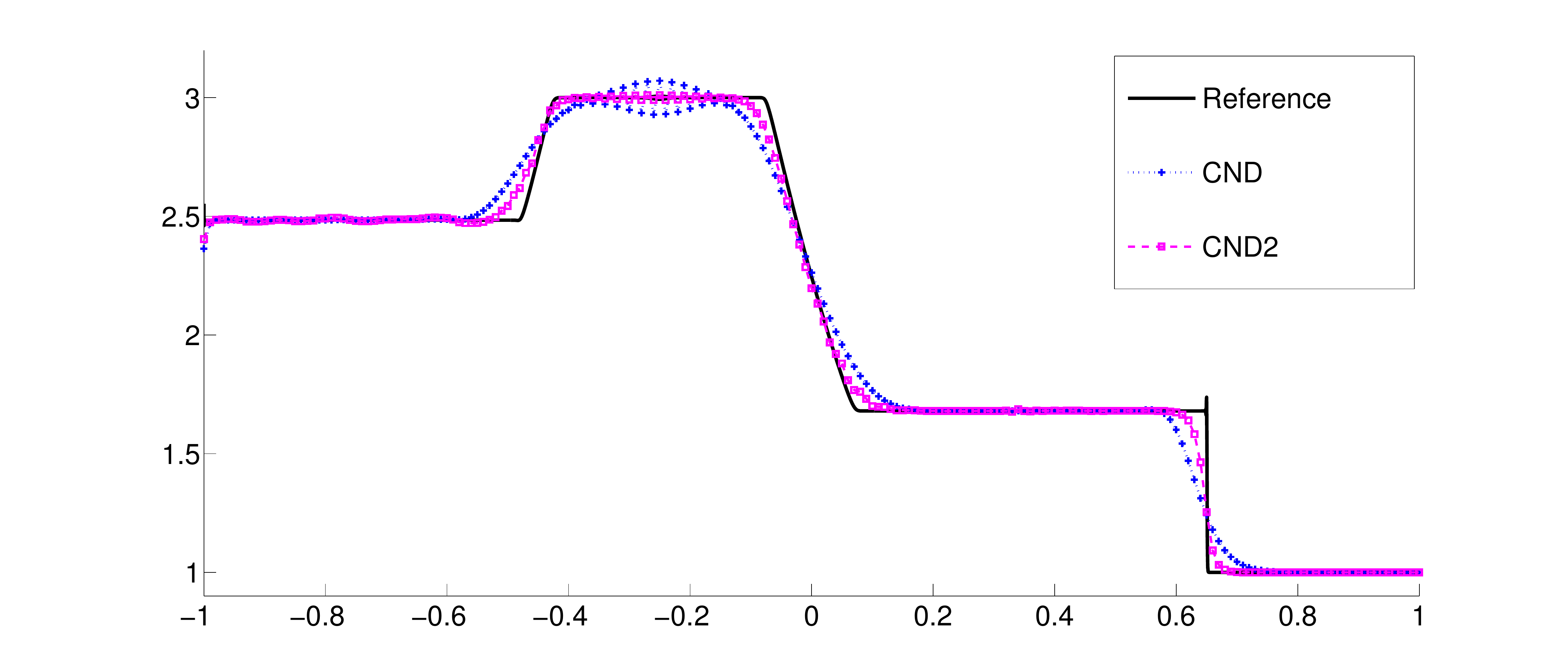}} 
\caption{Approximate solutions of the compressible Euler equations \eqref{eq:euler} with initial data \eqref{eq:eulinit} and boundary data \eqref{eq:eulbd}. We compare the CND and CND2 schemes on $200$ mesh points with a reference solution of the compressible Navier-Stokes equations \eqref{eq:cns} with $\kappa = \nu = 10^{-5}$.}
\label{fig:5}
\end{figure}
\par We compute approximate solutions of the Euler equations with  initial data \eqref{eq:eulinit} and boundary data \eqref{eq:eulbd} using the CND and CND2 schemes and show the results, obtained on a mesh of $200$ points, in figure \ref{fig:5}. The result shows that both the first and second order CND schemes approximate the physically relevant solution, computed as the limit of the compressible Navier-Stokes equations, quite well. The first-order scheme smears the discontinuities as well as generates oscillations. On the other hand, the second-order scheme is clearly sharper at discontinuities. Furthermore, it reduces the oscillations considerably. 
\section{Conclusion}
We consider the initial-boundary value problem for systems of conservation laws \eqref{eq:cl}. Since the work by Gisclon and Serre~\cite{G1,GS1} it is known that the solutions of the initial boundary value problem depend on the underlying viscous approximation~\eqref{eq:vcl}.  Different choices of viscosity operators can lead to different 
solutions for the limit system of conservation laws \eqref{eq:cl}. These results hold for both linear as well as non-linear systems. Even $2\times 2$, strictly hyperbolic, 
symmetrizable linear systems like the linearized shallow water equations \eqref{eq:lsws} show this behavior. 

This dependence of solutions on underlying small scale effects suggests that one should discretize 
the viscous approximation~\eqref{eq:vcl} directly. However, this is very expensive computationally on account of very low values of the viscosity parameter. 
Therefore, we need to design numerical schemes for the system of conservation laws \eqref{eq:cl} that converge to the physically relevant solutions i.e the limit of solutions of \eqref{eq:vcl} as $\epsilon \rightarrow 0$. Unfortunately, existing numerical schemes like the standard Godunov, Roe and HLL schemes might converge of the physically incorrect solution of the initial-boundary value problem. 

In this paper, we design a conservative finite difference scheme \eqref{eq:fds} with a numerical flux \eqref{eq:nf2} based on the following two ingredients:
\begin{itemize}
\item entropy conservative fluxes \eqref{eq:ecf};
\item numerical diffusion operators \eqref{eq:ndiff}.
\end{itemize}
Information about the underlying viscous approximation~\eqref{eq:vcl} is explicitly incorporated into the 
choice of the numerical diffusion operator. The resulting entropy stable schemes are shown (numerically) to 
converge to the limit solution, obtained from the underlying viscous approximation. 
Thus, we provide a numerical framework for computing solutions of the system of conservation laws 
that require explicit information about the underlying small scale effects. To the best of our knowledge, this is the first time such schemes have been constructed in the context of initial-boundary value problems. 

We present a set of numerical experiments for both the linearized shallow water and 
nonlinear Euler equations to demonstrate that our numerical schemes do converge to the limit solutions of the underlying 
eddy viscosity or Navier-Stokes viscosity, respectively. Second-order schemes are constructed and are shown to 
be superior to first-order schemes in terms of accuracy as well as in suppressing oscillations that might result 
from a lack of viscosity in some conservative variables. At the same time, these second-order schemes also 
converge to the physically relevant solutions. 

We concentrated on Dirichlet boundary conditions in one space dimension in this paper. 
Extensions to several space dimensions and to other interesting boundary conditions will be considered in a 
forthcoming paper.


\begin{thebibliography}{100}

\bibitem{Amadori}
D. Amadori
\newblock Initial-boundary value problems for nonlinear systems of conservation laws.
\newblock {\em NoDEA Nonlinear Differential Equations Appl.} 4, no. 1, 1-42, 1997.


\bibitem{AGSID1}
Adimurthi, S. Mishra, and G. D. Veerappa Gowda.
\newblock Optimal entropy solutions for scalar conservation laws with discontinuous flux.
\newblock {\em J. Hyperbolic. Diff. Eqns.,} 2(4), 787-838, 2005.

\bibitem{ANBIA}
F. Ancona and S. Bianchini
\newblock Vanishing viscosity solutions of hyperbolic systems of conservation laws with boundary
\newblock {\em ``WASCOM 2005''—13th Conference on Waves and Stability 
in Continuous Media, 13-21, World Sci. Publ., Hackensack, NJ} 2(4), 2006.

\bibitem{BIA:ARMA}
S. Bianchini
\newblock On the Riemann problem for non-conservative hyperbolic systems.
\newblock {\em Arch. Ration. Mech. Anal.} 166, no. 1, 1-26, 2003.


\bibitem{BIASPI:ARMA}
S. Bianchini and L.V. Spinolo
\newblock The boundary Riemann solver coming from the real vanishing viscosity approximation.
\newblock {\em Arch. Ration. Mech. Anal.} 191, no. 1, 1-96, 2009.


\bibitem{BIABRE:ANNALS}
S. Bianchini and A. Bressan 
\newblock Vanishing viscosity solutions of nonlinear hyperbolic systems. 
\newblock {\em Ann. of Math.}(2) 161, no. 1, 223-342, 2005. 


\bibitem{CLMP1}
M. J. Castro, P. LeFloch, M. L. Munoz Ruiz and C. Pares.
\newblock \emph{Why many theories of shock waves are necessary: convergence error in formally path-consistent schemes.}
\newblock J. Comput. Phys., 227 (17), 2008, 8107-8129.

\bibitem{DAF1}
C. Dafermos.
{\em Hyperbolic conservation laws in continuum physics. }
\newblock Third edition. Grundlehren der Mathematischen Wissenschaften 
[Fundamental Principles of Mathematical Sciences], 325. {\em Springer-Verlag}, Berlin, 2010.

\bibitem{DLF1}
F. Dubois and P. LeFloch.
\newblock Boundary conditions for non-linear hyperbolic systems.
\newblock {\em J. Differential Equations,} 71 (1), 93-122, 1988.


\bibitem{FM2}
U. S. Fjordholm and S. Mishra.
\newblock Accurate numerical discretizations of non-conservative hyperbolic systems.
\newblock \emph{M2AN Math. Model. Num. Anal} to appear. Research Report N. 2010--25, Seminar f\"ur Angewandte Mathmatik ETH Z\"urich, 2010.





\bibitem{FMT4}
U. S. Fjordholm, S. Mishra and E. Tadmor.
\newblock Arbitrary order accurate essentially non-oscillatory entropy stable schemes for systems of
conservation laws.
\newblock  \emph{Research Report N. 2011--39}, Seminar f\"ur Angewandte Mathmatik ETH Z\"urich, 2011.

\bibitem{G1}
M . Gisclon.
\newblock \'Etude des conditions aux limites pour un syst\`eme strictement 
hyperbolique, via l'approximation parabolique.
\newblock {\em J. Math. Pures Appl.}, 9 (75), 485-508, 1996.


\bibitem{GS1}
M . Gisclon and D. Serre.
\newblock \'Etude des conditions aux limites pour un syst\`eme strictement 
hyperbolique, via l'approximation parabolique.
\newblock {\em C. R. Acad. Sci. Paris}, 319 (4), 377-382, 1994.

\bibitem{Goodman}
J. Goodman.
\newblock Initial-Boundary Value Problems for Hyperbolic Systems of Conservation Laws
\newblock {\em PhD Thesis}, Standford University, 1983.


\bibitem{GST1}
S. Gottlieb, C. W. Shu and E. Tadmor.
\newblock High order time discretizations with strong stability property.
\newblock {\em SIAM. Review,} 43, 89 - 112, 2001.

\bibitem{IR1}
F. Ismail and P. L. Roe.
\newblock Affordable, entropy-consistent Euler flux functions II: Entropy production at shocks
\newblock \emph{Journal of Computational Physics} 228(15), volume 228, 5410–5436, 2009


\bibitem{KawashimaShizuta}
Kawashima, S. and Shizuta, Y. 
\newblock On the normal form of the symmetric hyperbolic-parabolic systems associated with the conservation laws.
\newblock {\em Tohoku Math. J.} (2) 40, no. 3, 449-464, 1988. 


\bibitem{JOLEF:ARMA}
K.T. Joseph and P.G. LeFloch.
\newblock Boundary layers in weak solutions of hyperbolic conservation laws.
\newblock {\em Arch. Ration. Mech. Anal.} 147, no. 1, 47-88, 1999. 


\bibitem{LF1}
P. G. LeFloch.
\newblock Kinetic relations for undercompressive shock waves: Physical, mathematical and numerical issues.
\newblock {\em Nonlinear partial differential equations and hyperblic wave phenomena},  237-272, Contemp. Math., 526, {\em Amer. Math. Soc, Providence, 2010}.

\bibitem{Lax}
P.D. Lax
\newblock Hyperbolic systems of conservation laws. {II}.
\newblock {\em Comm. Pure Appl. Math.} 10, 537-566, 1957.

\bibitem{LIU:JDE}
T.P. Liu
\newblock The Riemann problem for general systems of conservation laws.
\newblock {\em J. Differential Equations} 18, 218-234, 1975.

\bibitem{LIU:JMAA}
T.P. Liu
\newblock The entropy condition and the admissibility of shocks.
\newblock {\em J. Math. Anal. Appl.} 53, no. 1, 78-88, 1976.

\bibitem{LEV1}
R. J. LeVeque.
\newblock Finite volume methods for hyperbolic problems.
\newblock {\em Cambridge Univ. Press,} Cambridge, 2002.

\bibitem{SableTougeron}
M. Sabl\'e-Tougeron
\newblock M\'ethode de Glimm et probl\`eme mixte. 
\newblock {\em Ann. Inst. H. Poincar\'e Anal. Non Lin\'eaire} 10, no. 4, 423-443, 1993.

\bibitem{Serre:b1}
D. Serre. 
\newblock Systems of conservation laws. 1. Hyperbolicity, entropies, shock waves.
              Translated from the 1996 French original by I. N. Sneddon.
 \newblock {\em Cambridge Univ. Press}, Cambridge, 1999. 


\bibitem{Serre:b2}
 D. Serre. 
\newblock Systems of conservation laws. 2. Geometric structures, oscillations, and initial-boundary value
              problems.
              Translated from the 1996 French original by I. N. Sneddon.
 \newblock {\em Cambridge Univ. Press}, Cambridge, 2000. 

\bibitem{Serre:lecturenotes}
\newblock D. Serre. 
\newblock Discrete shock profiles: existence and stability. \emph{Hyperbolic systems of balance laws}, 79-158,
Lecture Notes in Math., 1911, 
\newblock {\em Springer}, Berlin, 2007.

\bibitem{TAD1}
E. Tadmor.
\newblock The numerical viscosity of entropy stable schemes for systems of conservation laws, I.
\newblock {\em Math. Comp.,} 49, 91-103, 1987.




\end{thebibliography}
\end{document}